\documentclass[11pt,a4paper]{amsart}

\usepackage{tikz}
\usepackage{tikz-cd}
\usepackage{url}
\usepackage[colorlinks=true,linkcolor=blue,citecolor=blue,urlcolor=blue]{hyperref}
\usepackage{pgfplots}   
\usepgfplotslibrary{polar}   
\pgfplotsset{compat=newest}   
\usepackage{graphicx}
\usepackage{subcaption}
\title{Inner functions associated to lifts of transcendental entire functions}
\date{\today}
\author{Eleni Betsakou}
\address{School of Mathematics and Statistics, The Open University, Walton Hall, Milton Keynes MK7 6AA, UK}
\email{eleni.betsakou@open.ac.uk}


\makeatletter
\@namedef{subjclassname@2020}{%
	\textup{2020} Mathematics Subject Classification}
\makeatother


\usepackage[utf8x]{inputenc}

\usepackage[UKenglish]{babel}

\usepackage{amssymb,mathrsfs,mathtools,amsmath,amsthm}

\usepackage[shortlabels]{enumitem}   

\numberwithin{equation}{section}

\usepackage{color}

\usepackage[a4paper]{geometry}

\newgeometry{twoside=false,textheight=\textheight+10ex,footskip=2.5\footskip,left=7em,right=7em}

\usepackage[all]{xy}   

\usepackage[backend=biber,style=alphabetic,doi=false,url=false,isbn=false,eprint=false,maxbibnames=99]{biblatex}
\DeclareFieldFormat{title}{#1}
\DeclareFieldFormat[article]{title}{#1}
\DeclareFieldFormat[incollection]{title}{#1}

\renewbibmacro{in:}{}

\newtheorem{thm}{Theorem}[section]
\newtheorem{prop}[thm]{Proposition}

\newtheorem {lemma}[thm]{Lemma}
\newtheorem {cor}[thm] {Corollary}

\theoremstyle{definition}

\newtheorem{defn}[thm]{Definition}  
\newtheorem{eg}[thm]{Example}
\newtheorem{rmk}[thm]{Remark}  
  
\newtheorem{obs}[thm]{Observation}

\newcommand{\C}{\mathbb{C}}
\newcommand{\R}{\mathbb{R}}
\newcommand{\N}{\mathbb{N}}
\newcommand{\Z}{\mathbb{Z}}

\newcommand{\D}{\mathbb{D}}  

\begin{document}
	\begin{abstract}
		Let $f$ be a transcendental entire function, $V$ be a simply connected Fatou component of $f,$ and $U$ be a Fatou component with $f(U)\subset V.$ There is a natural way to associate $f|_U$ to an inner function, namely a function $g_f:=\psi^{-1}\circ f\circ\varphi,$ where $\varphi:\D\to U$ and $\psi:\D\to V$ are Riemann maps. Inner functions have been used as a tool in the study of the iterates of transcendental entire, and more recently meromorphic, functions. However, there are only a few examples where associated inner functions have been calculated explicitly, with the case where $f$ has infinite degree in $U$ being the least well understood and more complicated.
		 
		In this paper, we introduce a general method for calculating associated inner functions to a wide class of entire functions arising as `lifts'. In particular, if $f$ is a lift of a transcendental entire function $h,$ we show that an inner function associated to $f|_U$ can be obtained by relating it to an inner function associated to $h|_G,$ where $G$ is the Fatou component that lifts to $U.$ This result significantly generalises the main part of a theorem by Evdoridou, Rempe and Sixmith, and can be applied to several functions that have been studied so far. In both finite- and infinite-degree settings, the results hold for forward-invariant Fatou components as well as for wandering domains.
	\end{abstract}
	\maketitle
	\section{Introduction}
	Let $f$ be a transcendental entire function, and denote by $f^n$ the $n$-th iterate of $f.$ Then the complex plane is divided into two complementary sets, the \emph{Fatou set} $F(f),$ which is the set of all points for which $\{f^n\}_n$ is a normal family in some neighbourhood, and the \emph{Julia set} $J(f),$ which is often referred to as the chaotic set. By definition, the Fatou set is open, and it is the union of its connected components, which are called \emph{Fatou components}. What is more, $F(f)$ is \emph{completely invariant}, meaning that $f(F(f))\subset F(f)$ and $f^{-1}(F(f))\subset F(f),$ and the same inclusions hold for $J(f).$ 
	
	Due to the complete invariance of the Fatou set, Fatou components are mapped by $f$ into Fatou components. A Fatou component $U$ can be \emph{periodic}, which means that there exists $n\in\N$ such that $f^n(U)\subset U;$ if $U$ is eventually mapped into a periodic Fatou component, then $U$ is called \emph{preperiodic}. A Fatou component which is not preperiodic is called a \emph{wandering domain}. Periodic Fatou components can be classified further; they can be \emph{attracting} or \emph{parabolic basins}, \emph{Siegel discs}, or \emph{Baker domains}. The latter do not exist for rational functions, and they are, by definition, periodic Fatou components in which the iterates converge to the essential singularity at $\infty.$ Also, a Fatou component $U$ is called \emph{forward-invariant} if $f(U)\subset U.$ More details on the classification of Fatou components can be found in \cite[Sections 4.1, 4.2]{Ber93}.
	
	Let $V\subsetneq\C$ be a simply connected domain and $U$ be a connected component of $f^{-1}(V).$ Then $U$ is also simply connected, so we can consider Riemann maps $\varphi:\D\to U$ and $\psi:\D\to V$ and define the function $g_f=\psi^{-1}\circ f\circ\varphi:\D\to\D$ (see Figure \ref{assinn}).
	\begin{figure}[h]
		\centering
		\includegraphics[width=0.55\linewidth]{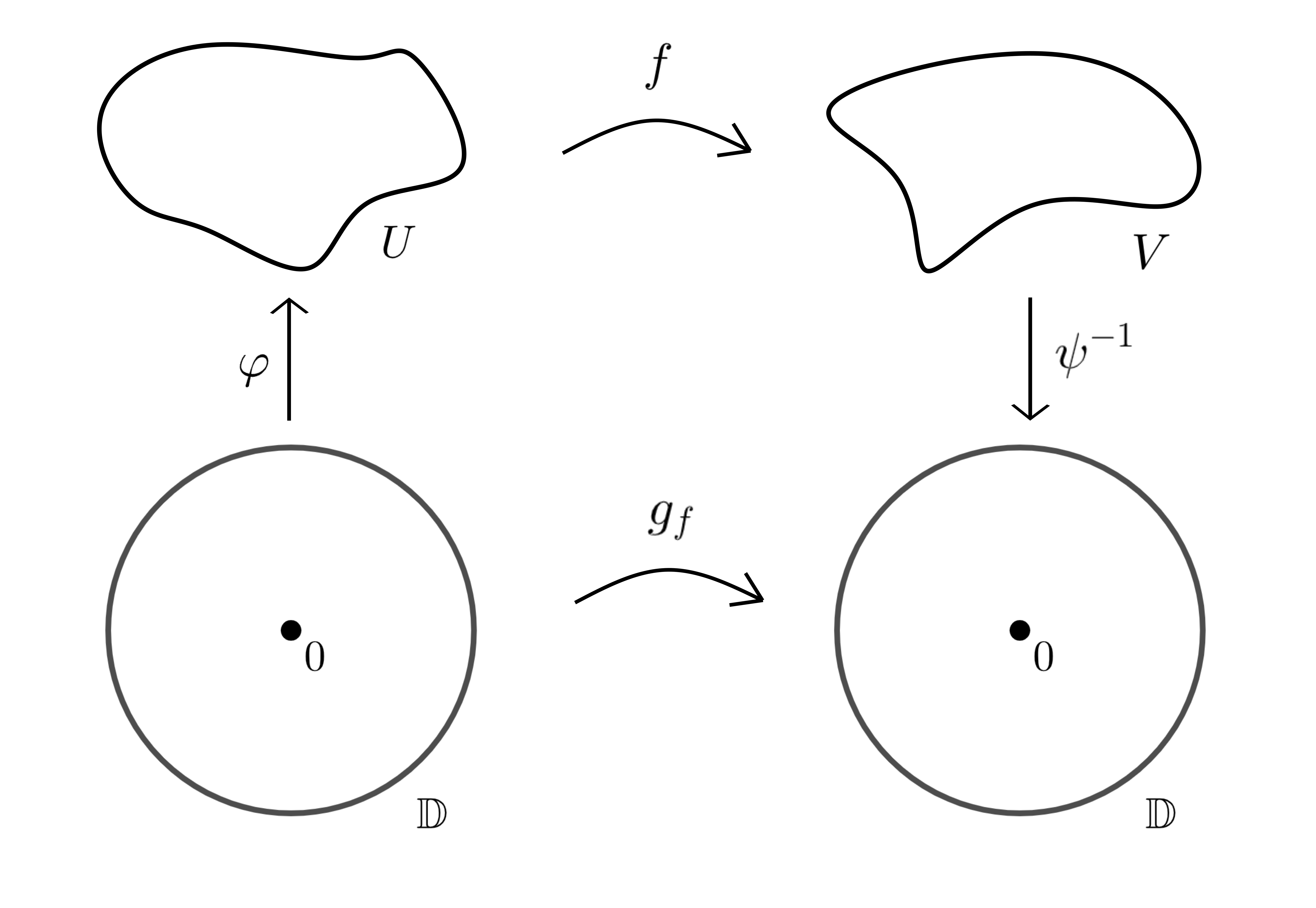}
		\caption{Constructing an associated inner function to $f|_U$}
		\label{assinn}
	\end{figure}	
	It can be shown that $g_f$ is an \emph{inner function}, that is, a holomorphic self-map of the unit disc for which radial limits exist almost everywhere and belong to the unit circle. When an inner function is constructed this way, it is called an \emph{inner function associated to} $f|_U.$ In the simpler case where $U=V,$ we can use the same Riemann map, and then $g_f$ is called a \emph{dynamically associated inner function to} $f|_U$ and it is unique up to a conformal conjugacy.
	
	There is a growing interest in inner functions associated to Fatou components. One reason for this is that it is often simpler to study a problem in the unit disc, using inner functions, and then transfer the results to the dynamical plane (see, for example, \cite{BarK07}, \cite{BeniEFRS}, \cite{BFXK17}, \cite{FJ25}). Another reason is that, in the case of transcendental entire functions, the associated inner functions that have been calculated so far behave relatively nicely near the boundary of the unit disc. This is not the case for inner functions in general, so the natural question that arises is if this nice behaviour holds for all inner functions that are associated to Fatou components, or under which hypotheses for the transcendental entire function and the Fatou component the associated inner function behaves nicely (see, for example, \cite[Theorem C]{JoF25}). 
	
	Finding an inner function associated to a Fatou component of a transcendental entire function can be a challenging task, and there are only a few examples for which such an inner function has been calculated explicitly. It seems that the first such example appears in \cite[§V]{topf}. There, Töpfer studied the function $f(z)=\sin z,$ which has two immediate parabolic basins at 0. Both basins have the same dynamically associated inner function, which is of the form $$g_f(z)=\frac{z^2+k}{kz^2+1},\,z\in\D,$$ where $k=\frac{1}{3}$ (see \cite[p. 463]{ERS20}). Another example is the family of functions $f_\lambda(z)=\lambda e^z,$ where $\lambda\in\C$ is such that $f_\lambda$ has a completely invariant attracting basin. Devaney and Goldberg studied this family in \cite{dg87} and found that a dynamically associated inner function to this basin is $$g_{f_\lambda}(z)=\exp\Big(i\,\frac{\mu+\bar{\mu}z}{1+z}\Big),\,z\in\D,$$ where $\mu$ belongs to the upper half-plane $\mathbb{H}$ and depends on $\lambda.$ Also, Baker and Domínguez showed in \cite[Section 5]{BD99} that the function $f(z)=z+e^{-z}$ has infinitely many forward-invariant Baker domains $\{U_k\}_{k\in\mathbb{Z}}$ such that $U_{k+1}=U_k+2\pi i,$ for all $k\in\mathbb{Z},$ and that $$g_f(z)=\frac{3z^2+1}{3+z^2},\,z\in\D,$$ which is a Blaschke product of degree 2, is an inner function dynamically associated to each one of these Baker domains. 
	
	In the first and third example mentioned in the previous paragraph, the associated inner function being a finite Blaschke product is not a mere coincidence; it follows from the fact that the function $f$ has finite degree in the Fatou component under consideration. The connection between the degree of a function in a Fatou component and the form of its associated inner function is a special case of \cite[Proposition 1.1]{ERS20}. 
	
	When the transcendental entire function $f$ has infinite degree in the Fatou component under consideration, finding an associated inner function is more complicated, because the inner function has infinite degree too. Such an inner function could be an infinite Blaschke product, a singular inner function (that is, an inner function with no zeros), or a product of a Blaschke product and a singular inner function (for more information and examples of inner functions, see §\ref{if}). In \cite{ERS20}, the authors studied this problem for some classes of transcendental entire functions, such as functions of the form $f_\lambda(z)=\lambda\sin z,\,\lambda\in(0,1)$ (\cite[Theorem 1.8]{ERS20}), $f(z)=P(z)e^{Q(z)},$ where $P,Q$ are polynomials (\cite[Corollary 6.2]{ERS20}), and $f_\lambda(z)=\lambda+z+e^{-z},\,\lambda>0$ (\cite[Theorem 1.9]{ERS20}). The latter functions are called Fatou functions, since $f(z)=1+z+e^{-z},\,z\in\C,$ was first studied by Fatou in \cite{Fatou} (see Figure \ref{figf}), and their Fatou set consists of a single component, which is a Baker domain. 
	\begin{figure}[h]
		\centering
		\includegraphics[width=0.4\linewidth]{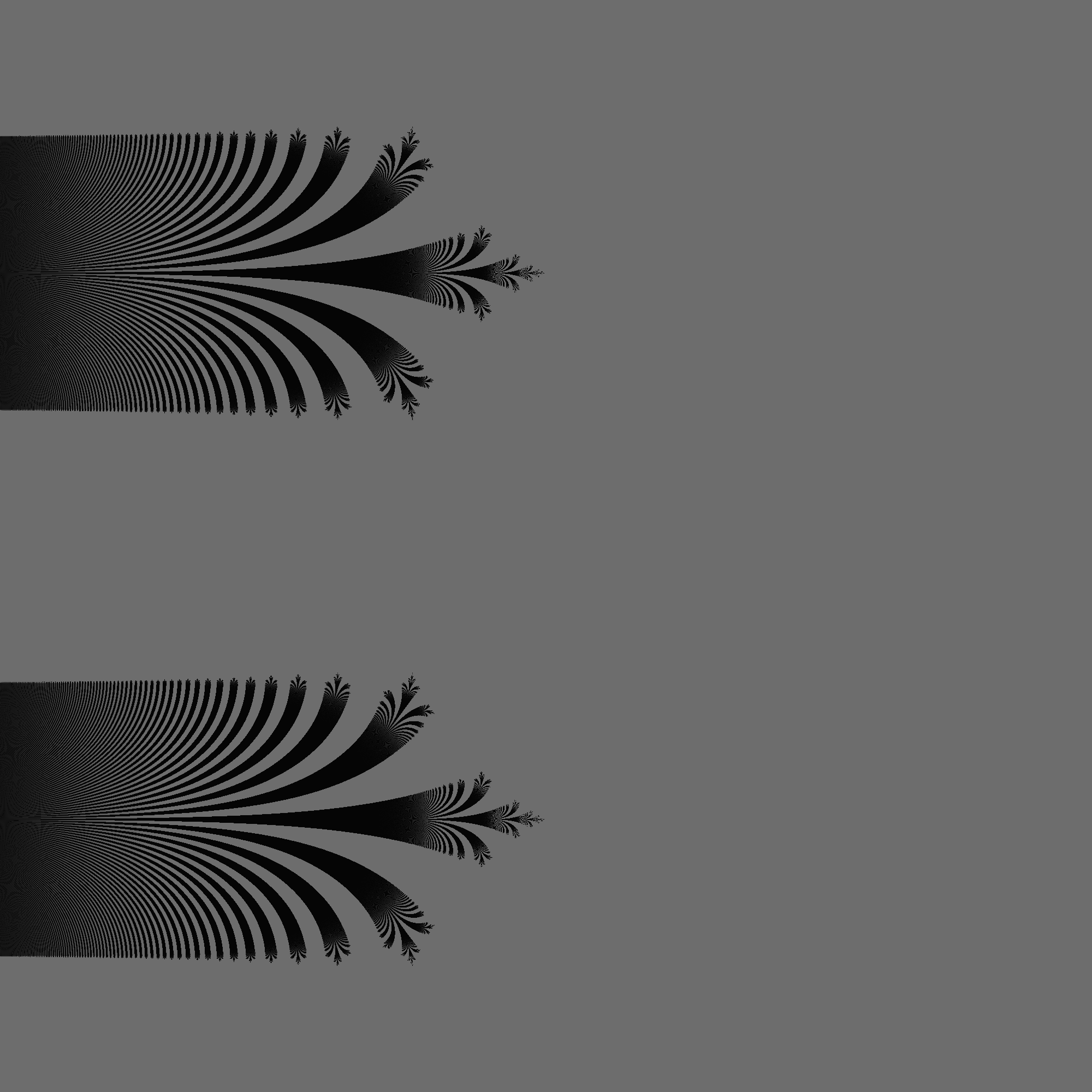}
		\caption{The Fatou (grey) and Julia (black) sets of the function $f(z)=1+z+e^{-z}.$}
		\label{figf}
	\end{figure}
	Compared with the aforementioned examples, the Fatou functions exhibit additional complexity; the set of singular values - the closure of the set of critical and finite asymptotic values - of $f_\lambda$ is not compactly contained in the Baker domain. This fact makes it difficult to calculate a dynamically associated inner function, since there is no general theory for such cases (compare with \cite[Theorem 1.2]{ERS20}). Evdoridou, Rempe and Sixmith showed that a dynamically associated inner function of the upper half-plane $\mathbb{H}$ to the Baker domain of $f_\lambda$ is $$g_\lambda:\mathbb{H}\to\mathbb{H};\;\;\; g_\lambda(z)=z-\lambda\frac{\cot z}{2},\,z\in\mathbb{H},$$ by using the periodicity of the set of fixed points of $f_\lambda.$ 
	
	A natural question that arises from the above discussion is whether there exists a general method that applies to as many of the functions mentioned so far as possible and yields their associated inner functions. By examining the transcendental entire functions discussed above, one can observe that several of them share a common feature: they are ‘lifts’ of transcendental entire functions. In fact, this is exactly what enables us to use the lifting method (\cite{Ber95}) in order to study their dynamics. Motivated by this observation, we develop a method of calculating an associated inner function for classes of functions that are lifts. Note that the lifting method was recently used in \cite[§5.3]{FJ25} to calculate an associated inner function to the attracting basin of a meromorphic function; in contrast to the method developed in this paper, the approach taken there passes from the lift to the original function.
	
	We denote by $\mathcal{L}$ the set of all transcendental entire functions $h:\C\to\C$ such that $h(\C^*)\subset\C^*,$ where $\C^*=\C\setminus\{0\}.$ Note that $h\in\mathcal{L}$ if and only if $h$ has the form 
	\begin{equation} \label{genf}
		h(w)=w^me^{H(w)}, \text{ for all } w\in\C,
	\end{equation}
	where $m\in\N\cup\{0\}$ and $H$ is a non-constant entire function. Let $h\in\mathcal{L}$ and $\pi_a(z)=e^{az},$ for all $z\in\C,$ where $a\in\C^*.$ We denote by $\mathcal{L}_{a,\,h}$ the set of all transcendental entire functions $f$ that satisfy $h\circ\pi_a=\pi_a\circ f$ in $\C.$ The functions in the class $\mathcal{L}_{a,\,h}$ are called the \emph{lifts} of $h.$
	
	The main theorem of this paper deals with the infinite-degree case and is stated below. Note that the assumptions (a)-(c) coincide with the hypotheses of \cite[Theorem 6.1]{ERS20} (see Theorem \ref{ers}).
	\begin{thm} \label{const}
		\textnormal{\textbf{(Construction Theorem)}} Let $h$ be a function of the form \eqref{genf} and assume that:
		\begin{enumerate}[(a), leftmargin=*, itemsep=0pt, topsep=0pt]
			\item $h$ has an (unbounded) forward-invariant Fatou component $G$ on which $h$ has infinite degree,
			\item there exists a point $b\in G,$ such that $h^{-1}(\{b\})\cap G$ contains exactly $p\ge0$ points, counting multiplicity, and
			\item an inner function dynamically associated to $h|_G,\,g_h,$ has a finite number $q\ge1$ of singularities on $\partial\D.$
		\end{enumerate}  
		If $f\in\mathcal{L}_{a,\,h},\,a\in\C^*,$ then exactly one of the following holds.
		\begin{enumerate}[(i), leftmargin=*, itemsep=0pt, topsep=0pt]
			\item $0\in G$ and $U:=\pi_a^{-1}(G)$ is connected. In this case, $b=0,\,p=m$ and $f|_U$ has a dynamically associated inner function of the form
			\begin{equation} \label{assf}
				g_f:\mathbb{H}\to\mathbb{H};\;\;\;g_f(z)=\sigma+mz+\frac{i}{2}\sum_{j=1}^{q}\Big(c_j\cdot\frac{e^{i\theta_j}+e^{2iz}}{e^{i\theta_j}-e^{2iz}}\Big),
			\end{equation} 
			where $\sigma,\,\theta_1,\dots,\theta_q$ are real numbers and $c_1,\dots,c_q$ are positive real numbers.
			\item $0\notin G$ and $\{U_k\}_{k\in\Z}$ are the components of $\pi_a^{-1}(G).$ Then $g_h$ is associated to $f|_{U_k},$ for each $k\in\Z,$ and it can be taken to be of the form   $$g_h:\D\to\D;\;\;\;g_h(z)=B(z)\cdot\exp\Big(-\sum_{j=1}^{q}\Big(c_j\cdot\frac{e^{i\theta_j}+z}{e^{i\theta_j}-z}\Big)\Big),$$ where $B$ is a finite Blaschke product of degree $p,\,\theta_j\in\R$ and $c_j>0,$ for all $j\in\{1,\dots,q\}.$
		\end{enumerate}			 
	\end{thm}
	\begin{rmk} 
		Let $h\in\mathcal{L}$ be such that $h=h_n\circ\dots\circ h_1,$ where $n\in\N$ and $h_1,\dots,\,h_n$ are transcendental entire functions that satisfy the hypotheses of \cite[Theorem 6.1]{ERS20} for a Fatou component $G.$ Assume, also, that $G$ is a Fatou component of $h.$ Then, $g_h=g_{h_n}\circ\dots\circ g_{h_1}$ is a dynamically associated inner function to $h|_G,$ where $g_{h_j}$ is an inner function dynamically associated to $h_j|_G,$ for $j\in\{1,\dots,\,n\},$ and is conjugate to the inner function that is calculated explicitly in \cite[Theorem 6.1]{ERS20}. If $f\in\mathcal{L}_{a,\,h},\,a\in\C^*,$ and $U$ is a component of $\pi_a^{-1}(G),$ the proof of Theorem \ref{const} can be applied to $h,$ even if $h$ does not satisfy the hypotheses of this theorem (for instance, $g_h$ can have infinitely many singularities on $\partial\D$). Examples like this can arise if we compose a function $h\in\mathcal{L}$ with itself finitely many times.
	\end{rmk}
	Theorem \ref{const}(i) has a corollary that can be applied to a large class of functions. This includes the Fatou functions (\cite[Theorem 1.9]{ERS20}), and functions of the form $f(z)=mz+\lambda+e^{-z},$ where $m\in\N,\,m\ge2,$ and $\lambda\in\C$ depends on $m$ (\cite[p. 12]{RRS10}). In particular, concerning the Fatou functions, the proof of \cite[Theorem 1.9]{ERS20}, apart from calculating the constant, can be deduced from the following corollary, which is proved in Section \ref{inf}.
	\begin{cor} \label{infpoly}
		Suppose $Q$ is a polynomial of degree $\deg Q\ge1.$ Suppose also that the function $$h(w)=w^me^{Q(w)},$$	where $m\in\N\cup\{0\},$ has an (unbounded) forward-invariant Fatou component $G$ containing the origin, on which $h$ has infinite degree. If $f\in\mathcal{L}_{a,\,h},\,a\in\C^*,$ then $U:=\pi_a^{-1}(G)$ is connected and $f|_U$ has a dynamically associated inner function of the form \eqref{assf}, where $q\le\deg Q.$
	\end{cor}
	In forthcoming work, we investigate an example of a function $h\in\mathcal{L},$ satisfying hypotheses (a)-(c) of Theorem \ref{const}, for which $0\notin G.$ When a function has these properties, then its lifts will have infinite degree in the Fatou components $\{U_k\}_{k\in\Z}$ (see Proposition \ref{degreesn0}). By choosing appropriate lifts, we obtain functions for which $\{U_k\}_{k\in\Z}$ contains a sequence of wandering domains. Such functions will also provide concrete examples of entire functions with sequences of wandering domains in which the functions have infinite degree. To the best of our knowledge, there are currently no such examples in the literature.
	
	We use the idea of lifting transcendental entire functions to prove the following two results for the finite-degree case.
	\begin{thm} \label{uds}
		Let $h\in\mathcal{L},\,f\in\mathcal{L}_{a,\,h},\,a\in\C^*,$ and suppose that $h$ has a forward-invariant Fatou component $G$ on which $h$ has finite degree $d.$ Also, let $g_h$ be an inner function dynamically associated to $h|_G.$ Then, exactly one of the following holds.
		\begin{enumerate}[(i), leftmargin=*, itemsep=0pt, topsep=0pt]
			\item $0\in G$ and $U:=\pi_a^{-1}(G)$ is connected. In this case, $f|_U$ is univalent and hence, has a dynamically associated inner function $g_f$ which is a conformal automorphism of the unit disc.
			\item $0\notin G$ and $\{U_k\}_{k\in\Z}$ are the components of $\pi_a^{-1}(G).$ Then $g_h$ is associated to $f|_{U_k},$ for all $k\in\Z,$ and it is a finite Blaschke product of degree $d.$ 
		\end{enumerate}
	\end{thm}
	A corollary to Theorem \ref{uds} applies to a class of functions that contains examples which have already been studied, such as $f(z)=z+e^{-z}$ (\cite[Section 5]{BD99}) and the function $f(z)=z-1+e^{-z},$ which was first studied by Herman in \cite{herman} (see, also, \cite[Example 5.1]{B84}). In Example \ref{tex}, we use Corollary \ref{finpoly} to calculate an inner function dynamically associated to Baker domains of a function which was examined in \cite[Example 4]{FH06}, as an example of an entire map with a sequence of doubly-parabolic Baker domains of degree 3.
	\begin{cor} \label{finpoly}
		Suppose $Q$ is a polynomial of degree $\deg Q\ge1.$ Suppose also that the function $$h(w)=w^me^{Q(w)},$$ where $m\in\N\cup\{0\},$ has a forward-invariant Fatou component $G,$ not containing 0, and $g_h$ is a dynamically associated inner function to $h|_G.$ Then $g_h$ is a finite Blaschke product of degree $d=m+\deg Q,$ and it is associated to $f|_{U_k},$ for each $k\in\Z,$ where $f\in\mathcal{L}_{a,\,h},\,a\in\C^*,$ and $\{U_k\}_{k\in\Z}$ are the components of $\pi_a^{-1}(G).$
	\end{cor}	
	In Theorem \ref{const}(ii), Theorem \ref{uds}(ii) and Corollary \ref{finpoly} we do not assume that the Fatou components $\{U_k\}_{k\in\Z}$ are forward-invariant. This is because there are cases where they are wandering domains or preimages of a forward-invariant Fatou component. For example, in \cite[Example 5.1]{B84}, Baker studied the function $f(z)=z-1+e^{-z},$ which is a function of Herman's type (see \cite{herman}), and a lift of the function $h(t)=te^{1-t}.$ He showed that it has infinitely many forward-invariant super-attracting basins $\{D_n\}_{n\in\Z},$ with $D_{n+1}=D_n+2\pi i,$ for all $n\in\Z,$ and then, he considered the function $g(z)=2\pi i+f(z),$ which satisfies $g(D_n)=D_{n+1},$ for all $n\in\Z,$ and thus is an example of a function with wandering domains. In this setting, we show the following result for the inner functions associated to such Fatou components.
	\begin{prop} \label{wandinn}
		Let $f$ be a transcendental entire function with the property $$f(z\pm c)=f(z)\pm mc, \text{ for all } z\in\C,$$ for some $m\in\N\cup\{0\}$ and $c\in\C\setminus\{0\}.$ Also, assume that $\{U_k\}_{k\in\Z}$ are simply connected Fatou components of $f$ such that
		\begin{enumerate}[(a), leftmargin=*, itemsep=0pt, topsep=0pt]
			\item $U_k+c=U_{k+1},$ for all $k\in\Z,$ and
			\item there exists $k_0\in\Z$ such that $f(U_0)\subset U_{k_0}.$
		\end{enumerate}   
		Let $\hat{f}(z)=f(z)+lc,$ for all $z\in\C,$ where $l\in\Z.$ Then:
		\begin{enumerate}[(i), leftmargin=*, itemsep=0pt, topsep=0pt]
			\item $F(\hat{f})=F(f).$ 
			\item For any $j,\,j'\in\Z,$ an inner function is associated to $f|_{U_j}$ if and only if it is associated to $\hat{f}|_{U_{j'}}.$ 
		\end{enumerate} 
	\end{prop}
	Finally, in \cite[4. Examples]{FH06}, Fagella and Henriksen presented a variety of functions with different types of Baker domains, and with different degrees in these Baker domains. In particular, in \cite[Example 1]{FH06}, they studied an example of a univalent hyperbolic Baker domain, while \cite[Example 2]{FH06} is an example of a univalent simply-parabolic Baker domain. The natural question they raised is whether there exist hyperbolic or simply-parabolic Baker domains of finite degree greater than 1 (\cite[table p. 385]{FH06}). We show that such examples cannot occur when the Baker domain is the lift of a forward-invariant Fatou component $G$ of a function $h$ of the form \eqref{genf}, with $m\ge1$ and with $0\in G.$ Under the same assumptions, we also prove that in the case where $G$ is lifted to a doubly-parabolic Baker domain, the lifts of $h$ always have infinite degree in this Baker domain. When $0\notin G,$ however, the case of doubly-parabolic Baker domains is different, as Fagella and Henriksen have already provided examples of this type with finite-degree (see \cite[Example 3, Example 4]{FH06}). We summarise the above results in the following proposition, and in Section \ref{exampl} we calculate the dynamically associated inner functions to two functions with infinite degree in their hyperbolic Baker domains. Example \ref{ex1} is a function examined in \cite[p. 69]{RS99} and \cite[Section 4]{RRS10}, while Example \ref{ex2} was studied in \cite[Example 3.6]{barg}.
	\begin{prop} \label{genrem}
		Let $h$ be of the form \eqref{genf}, with $m\ge1,$ and $G$ be a forward-invariant Fatou component of $h,$ with $0\in G.$ Also, let $f\in\mathcal{L}_{a,\,h},\,a\in\C^*,$ and $U=\pi_a^{-1}(G).$ Then, exactly one of the following holds.
		\begin{enumerate}[(a), leftmargin=*, itemsep=0pt, topsep=0pt]
			\item $m=1$ and \textnormal{Re}$H(0)=0.$ In this case, $U$ is a simply-parabolic Baker domain and $f|_U$ is univalent.
			\item $m=1$ and \textnormal{Re}$H(0)<0.$ In that case, $U$ is a doubly-parabolic Baker domain and $f|_U$ has infinite degree.
			\item $m\ge2.$ Then $U$ is a hyperbolic Baker domain, and $f|_U$ is either univalent or it has infinite degree.
		\end{enumerate}
	\end{prop}
	\textbf{Structure of the paper.} In Section \ref{prel}, we review some preliminary concepts, such as inner functions and the lifting method, we establish basic results concerning the components which a Fatou component is lifted to, and we also prove Proposition \ref{genrem}. In Section \ref{wdinn} we prove Proposition \ref{wandinn}, while Section \ref{inf} is devoted to the infinite-degree case, where we prove Theorem \ref{const} and Corollary \ref{infpoly}. In Section \ref{fin} we present the proofs of Theorem \ref{uds} and Corollary \ref{finpoly}, which concern the finite-degree case. Finally, in Section \ref{exampl} we discuss two examples of the infinite-degree case, in which Theorem \ref{const}(i) can be applied, as well as an example of the finite-degree case.
	
	\textbf{Acknowledgments.} I am deeply grateful to my supervisor, Vasiliki Evdoridou, for her ideas, guidance and support, and to my second supervisor, Gwyneth Stallard, for her valuable feedback and careful reading of this paper. Also, I am much obliged to Anna Jové for useful discussions in Barcelona and for her insightful remarks on the paper, and to David Martí-Pete for his interesting suggestions and the picture of Fatou's function. Finally, I wish to thank Charalampos Betsakos for providing the software used to generate the Julia set pictures in Section \ref{exampl}.
	\section{Preliminaries} \label{prel}
	\subsection{Inner functions} \label{if}
	As mentioned in the introduction, inner functions are an important tool when working with simply connected Fatou components. Here, we include some basic examples of inner functions and briefly describe the behaviour of inner functions near the boundary of the unit disc.
	
	The simplest examples of inner functions are \emph{Blaschke products}, that is, functions of the form $$B(z)=e^{i\theta}\prod_{n=1}^{d}\frac{|a_n|}{a_n}\cdot\frac{a_n-z}{1-\bar{a}_nz},\,z\in\D,$$ where $\theta\in\mathbb{R},\,d\in\mathbb{N}\cup\{+\infty\}$ and $\{a_n\}_{n=1}^d\subset\mathbb{D}$ is such that $\sum_{n=1}^{d}(1-|a_n|)<+\infty.$ When $a_n=0,$ we interpret the corresponding term in the product as $z.$ If $d\in\mathbb{N},$ then $B$ is called a \emph{finite Blaschke product} of degree $d.$ We also allow $d=0;$ in this case, $B(z)=e^{i\theta}$ is a constant inner function, which we regard as a finite Blaschke product of degree 0. If $d=+\infty,$ then $B$ is called an \emph{infinite Blaschke product}. There are non-constant inner functions that do not have zeros, the so-called \emph{singular inner functions}. One of the simplest singular inner functions is $$S(z)=\exp\frac{z+1}{z-1},\,z\in\mathbb{D}.$$ 
	\begin{rmk}
		We often use inner functions of the upper half-plane, $\mathbb{H},$ because they are sometimes easier to work with. A map $g:\mathbb{H}\to\mathbb{H}$ is called an inner function of the upper half-plane if the map $\tilde{g}=M^{-1}\circ g\circ M:\D\to\D,$ where $M:\D\to\mathbb{H}$ is a conformal homeomorphism from the unit disc onto the upper half-plane, is an inner function in the usual sense.
	\end{rmk}	
	Inner functions can have a rather irregular boundary behaviour. 
	\begin{defn}
		A point $\zeta\in\partial\mathbb{D}$ is called a \emph{singularity} of an inner function $g$ if $g$ cannot be extended holomorphically to any neighbourhood of $\zeta$ in $\mathbb{C}.$ 
	\end{defn}
	\subsection{The lifting method}
	In \cite{Ber95}, Bergweiler extensively studied a method to construct examples of transcendental entire functions, which at the same time provides the necessary tools to study the dynamics of the functions constructed with this method. We will describe this process and we will often refer to it as ‘the lifting method’.
	
	Recall that $h\in\mathcal{L}$ if and only if it has the form \eqref{genf}.  
	\begin{obs} \label{zeroav}
		If $h$ is of the form \eqref{genf}, then 0 is always a finite asymptotic value of $h.$ Indeed, $h$ takes the value 0 $m$ times (counting multiplicity), so Iversen's Theorem (see \cite{ivers}, \cite[Section 5, p. 371]{BE95}) implies that 0 is an asymptotic value of $h.$ 
	\end{obs}
	Next, let $h\in\mathcal{L}$ and $\pi_a(z)=e^{az},$ for all $z\in\C,$ where $a\in\C^*$.\footnote{Bergweiler used the exponential function in his paper, but we can obtain an analogous result with the function $z\mapsto e^{az}.$ In this paper we use the latter function for greater generality.} Then there exists a transcendental entire function $f$ such that  
	\begin{equation} \label{berl}
		h\circ\pi_a=\pi_a\circ f
	\end{equation}
	in $\C.$ The function $f$ belongs to $\mathcal{L}_{a,\,h}$ and it is unique up to an additive constant, which is an integer multiple of $\frac{2\pi i}{a}.$ It is easy to see that if $h$ has the form \eqref{genf} and $a\in\C^*,$ then $f\in\mathcal{L}_{a,\,h}$ if and only if $f$ has the form 
	\begin{equation} \label{formf}
		f(z)=mz+\frac{1}{a}H(e^{az})+\frac{2k\pi i}{a}, \text{ for all } z\in\C,
	\end{equation}
	where $k\in\Z.$ Note that if $f$ is of the form \eqref{formf}, then
	\begin{equation} \label{proli}
		f\Big(z+\frac{2l\pi i}{a}\Big)=f(z)+m\cdot\frac{2l\pi i}{a},\,\text{ for all } z\in\C \text{ and } l\in\Z.
	\end{equation} 
	
	Let $h\in\mathcal{L},\,f\in\mathcal{L}_{a,\,h},\,a\in\C^*.$ Then it follows from \cite[Theorem p. 2]{Ber95} that 
	\begin{equation} \label{fh}
		F(f)=\pi_a^{-1}(F(h)),
	\end{equation}
	and the same relation holds for the Julia sets of these functions. The advantage of this method is that it allows us to first study the dynamics of $h,$ which might be relatively simple, and then use \eqref{fh} to study the dynamics of $f.$ 
	\subsection{Lifting a Fatou component} 
	Let $h\in\mathcal{L},\,f\in\mathcal{L}_{a,\,h},\,a\in\C^*,$ and $G$ be a Fatou component of $h.$ Then it is easy to see, using \eqref{fh}, that every component of $\pi_a^{-1}(G)$ is a Fatou component of $f.$ What is more, we wish to determine how many components $G$ is lifted to, so we prove the following proposition, which is stated in a more general setting (the first implication is mentioned in \cite[p. 28]{barg}).
	\begin{prop} \label{liftg}
		Let $G\subset\C$ be a domain and $\pi_a(z)=e^{az},$ for all $z\in\C,$ where $a\in\C^*.$ If $0\in G,$ then $U:=\pi_a^{-1}(G)$ is connected. \\
		If, additionally, $G$ is simply connected, then the converse is also true.
	\end{prop}
	\begin{proof}
		First suppose that $0\in G,$ so there exists $r>0$ such that $D(0,r)\subset G.$ Then $H:=\pi_a^{-1}(D(0,r))=\{z=x+iy\in\C: x\cdot\text{Re}a-y\cdot\text{Im}a<\log r\}$ is a connected subset of $U.$ Let $z\in U\setminus H.$ Then $w:=\pi_a(z)\in G\setminus D(0,r)$ can be connected to a point $w'\in D(0,r)\setminus\{0\}$ by a curve $\gamma\subset G\setminus\{0\}$ (see Figure \ref{0in}). The unique lift $\tilde{\gamma}$ of $\gamma$ that starts at $z$ is contained in $U$ and ends at a point $z'$ inside $H.$ This shows that $U$ is connected.
		\begin{figure}[h]
			\centering
			\includegraphics[width=0.7\linewidth]{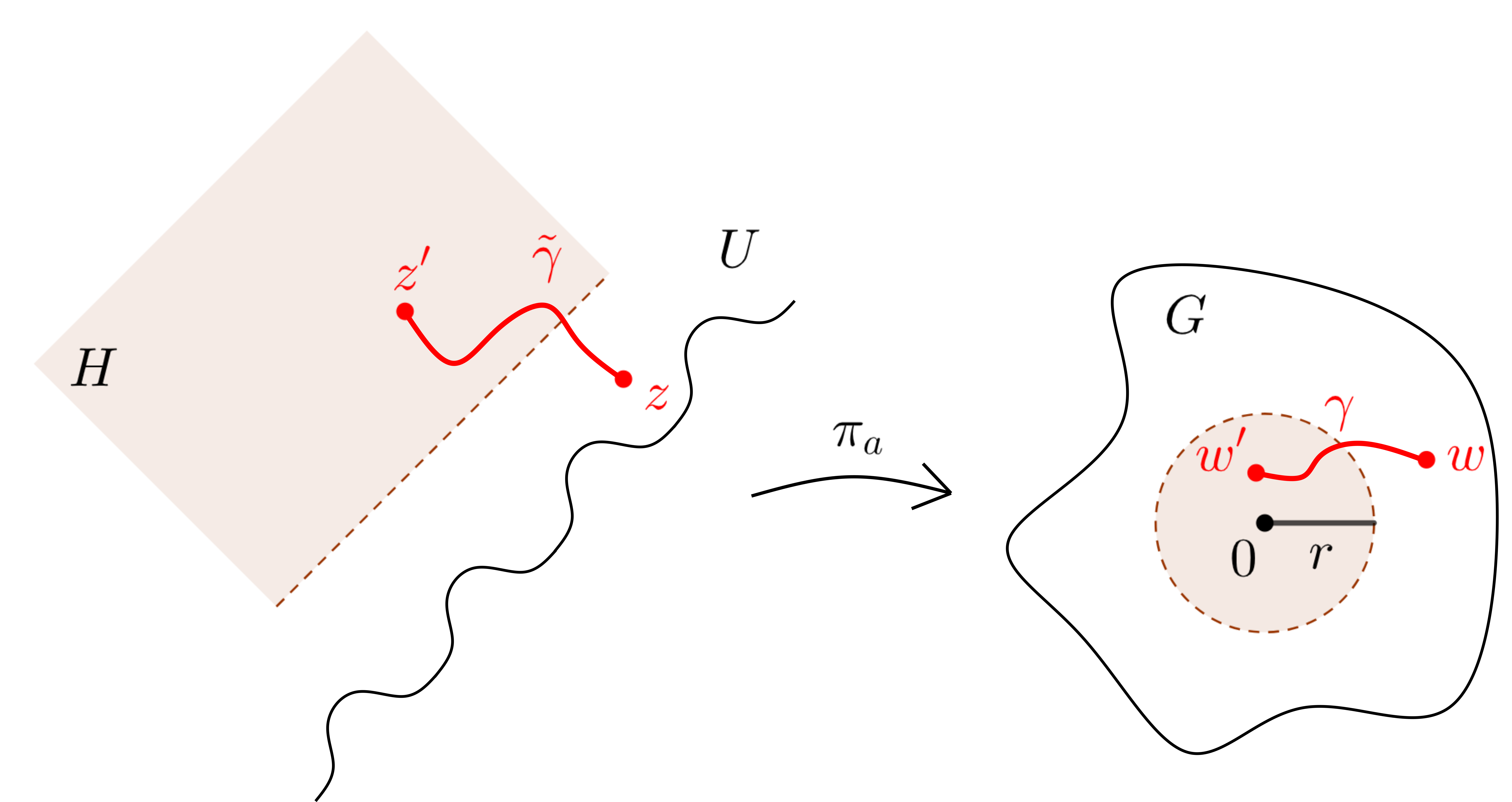}
			\caption{Proving that $U$ is connected}
			\label{0in}
		\end{figure}
		
		For the converse statement, suppose that $G$ is simply connected and $U:=\pi_a^{-1}(G)$ is connected. The definition of $U$ implies that it contains all the preimages under $\pi_a$ of any point in $G$ (except for 0, when $0\in G$). Since there are infinitely many such preimages, we deduce that $\pi_a$ is not injective in $U.$ Now, if $0\notin G,$ then the restriction $\pi_a|_U: U\to G$ is a covering map, because $\pi_a:\C\to\C^*$ is a covering map and $G$ is a connected subset of $\C^*.$ Thus, $\pi_a|_U$ has the curve lifting property (for the definition of the curve lifting property and more details on covering maps see, for example, \cite[Section 12.1, Definition 12.5, Lemma 12.6]{zakeri}), and it is also a local homeomorphism. Since $G$ is simply connected, it follows from \cite[Corollary 12.10]{zakeri} that $\pi_a$ is a homeomorphism in $U;$ a contradiction. Thus, $0\in G.$
	\end{proof}
	Proposition \ref{liftg} has an application to functions in the class $\mathcal{L}.$
	\begin{cor} \label{liftgl}
		Let $h\in\mathcal{L},\,G$ be a Fatou component of $h$ and $\pi_a(z)=e^{az},\,z\in\C,$ where $a\in\C^*.$ Then exactly one of the following holds.
		\begin{enumerate}[(i), leftmargin=*, itemsep=0pt, topsep=0pt]
			\item $0\in G$ and $U:=\pi_a^{-1}(G)$ is connected.
			\item $0\notin G$ and the connected components of $\pi_a^{-1}(G)$ form a sequence $\{U_k\}_{k\in\Z}$ such that $U_{k+1}=U_k+\frac{2\pi i}{a},$ for all $k\in\Z.$
		\end{enumerate} 
	\end{cor} 
	\begin{proof}
		Part (i) is the first part of Proposition \ref{liftg}.
		
		Regarding part (ii), assume that $0\notin G,$ and recall from Observation \ref{zeroav} that 0 is always an asymptotic value for functions in $\mathcal{L}.$ By \cite[Corollary to Theorem 3.1]{B84}, all the Fatou components of $h$ are simply connected, and so the second part of Proposition \ref{liftg} implies that $\pi_a^{-1}(G)$ is disconnected. The rest follows from the fact that $\pi_a$ is $\frac{2\pi i}{a}$-periodic in $\C.$
	\end{proof}
	\begin{rmk} \label{scu}
		\begin{enumerate}[(1), leftmargin=*, itemsep=0pt, topsep=0pt]
			\item If $h\in\mathcal{L},\,f\in\mathcal{L}_{a,\,h},\,a\in\C^*,$ and $G$ is a forward-invariant Fatou component of $h,$ then the components of $\pi_a^{-1}(G)$ are simply connected. Indeed, if $0\in G,$ then \eqref{berl} implies that $U:=\pi_a^{-1}(G)$ is a forward-invariant Fatou component of $f$ and hence, it is simply connected by \cite[Theorem 3.1]{B84}. If $0\notin G$ and $\{U_k\}_{k\in\Z}$ are the components of $\pi_a^{-1}(G),$ then $\pi_a$ is a conformal map from each $U_k$ onto $G.$ Indeed, for $k\in\Z$ fixed, if $z_1,\,z_2\in U_k$ are such that $\pi_a(z_1)=\pi_a(z_2),$ then $z_1=z_2+\frac{2l\pi i}{a}\in U_k\cap U_{k+l},$ for some $l\in\Z,$ and so $l=0,$ which establishes the injectivity of $\pi_a$ in $U_k.$ Since $G$ is simply connected, we conclude that each $U_k$ is simply connected.
			\item Let $h\in\mathcal{L},\,a\in\C^*,$ and suppose that $h$ has a forward-invariant Fatou component $G,$ such that $0\notin G.$ If $U_k,$ where $k\in\Z,$ is a connected component of $\pi_a^{-1}(G),$ then there exists an entire function $\hat{f}\in\mathcal{L}_{a,\,h}$ such that $\hat{f}(U_k)\subset U_k.$ Indeed, for $f\in\mathcal{L}_{a,\,h},$ \eqref{berl} implies that $f(U_k)\subset U_{k+l},$ for some $l\in\Z.$ If $l\ne0,$ consider the function $\hat{f}(z)=f(z)-l\cdot\frac{2\pi i}{a},$ for all $z\in\C,$ which belongs to $\mathcal{L}_{a,\,h}$ and satisfies $\hat{f}(U_k)\subset U_k.$
		\end{enumerate}		
	\end{rmk}
	Having now a complete picture of the components to which a Fatou component of $h$ is lifted, we can move on to compare the degrees of $h$ and its lifts in these components. We distinguish between two cases: $0\in G$ and $0\notin G.$ For the case where $0\in G,$ we first need the following result.
	\begin{lemma} \label{fdegf}
		Let $h$ be of the form \eqref{genf}, $f\in\mathcal{L}_{a,\,h},\,a\in\C^*,$ and $G$ be a forward-invariant Fatou component of $h,$ with $0\in G.$ If $\deg f|_U<+\infty,$ where $U=\pi_a^{-1}(G),$ then $m\ge1$ and $\deg f|_U=1.$  
	\end{lemma}
	\begin{proof}
		First, if $m=0,$ then $f$ is $\frac{2\pi i}{a}$-periodic by \eqref{proli}, and so it has infinite degree in $U$ (since $U+\frac{2\pi i}{a}=U$); a contradiction to the hypothesis. Therefore, $m\ge1.$ Next, one can observe that $f$ does not have critical points in $U.$ Indeed, if $z\in U$ is a critical point of $f,$ then $z+\frac{2k\pi i}{a}\in U,$ and \eqref{proli} implies that $z+\frac{2k\pi i}{a}$ is a critical point of $f,$ for all $k\in\Z,$ hence $f$ has infinitely many critical points in $U.$ This implies that $f|_U$ has infinite degree (see \cite[Proposition 2.8]{BerFR15}); a contradiction. Thus, $f|_U$ is a local homeomorphism and, by \cite[Theorem 12.33]{zakeri}, $f|_U:U\to U$ is a covering map. It follows from the Riemann-Hurwitz formula (see \cite[Theorem 12.47]{zakeri}) that $\deg f|_U=1.$ 
	\end{proof}
	\noindent The following proposition shows the relation between the degree of $h$ and the degree of its lifts, for the case where $0\in G.$
	\begin{prop} \label{degrees0}
		Let $h$ be of the form \eqref{genf}, $f\in\mathcal{L}_{a,\,h},\,a\in\C^*,$ and $G$ be a forward-invariant Fatou component of $h,$ with $0\in G.$ Then $h|_G$ has infinite degree if and only if $f|_U$ has infinite degree, where $U=\pi_a^{-1}(G).$ In particular, if $h|_G$ has finite degree, then $\deg h|_G=m\ge1,$ and $\deg f|_U=1.$ 
	\end{prop}
	\begin{proof}
		First assume that $h|_G$ has infinite degree. If $f|_U$ has finite degree, then Lemma \ref{fdegf} implies that $m\ge1$ and $\deg f|_U=1.$ 
		\begin{figure}[h]
			\centering
			\includegraphics[width=0.7\linewidth]{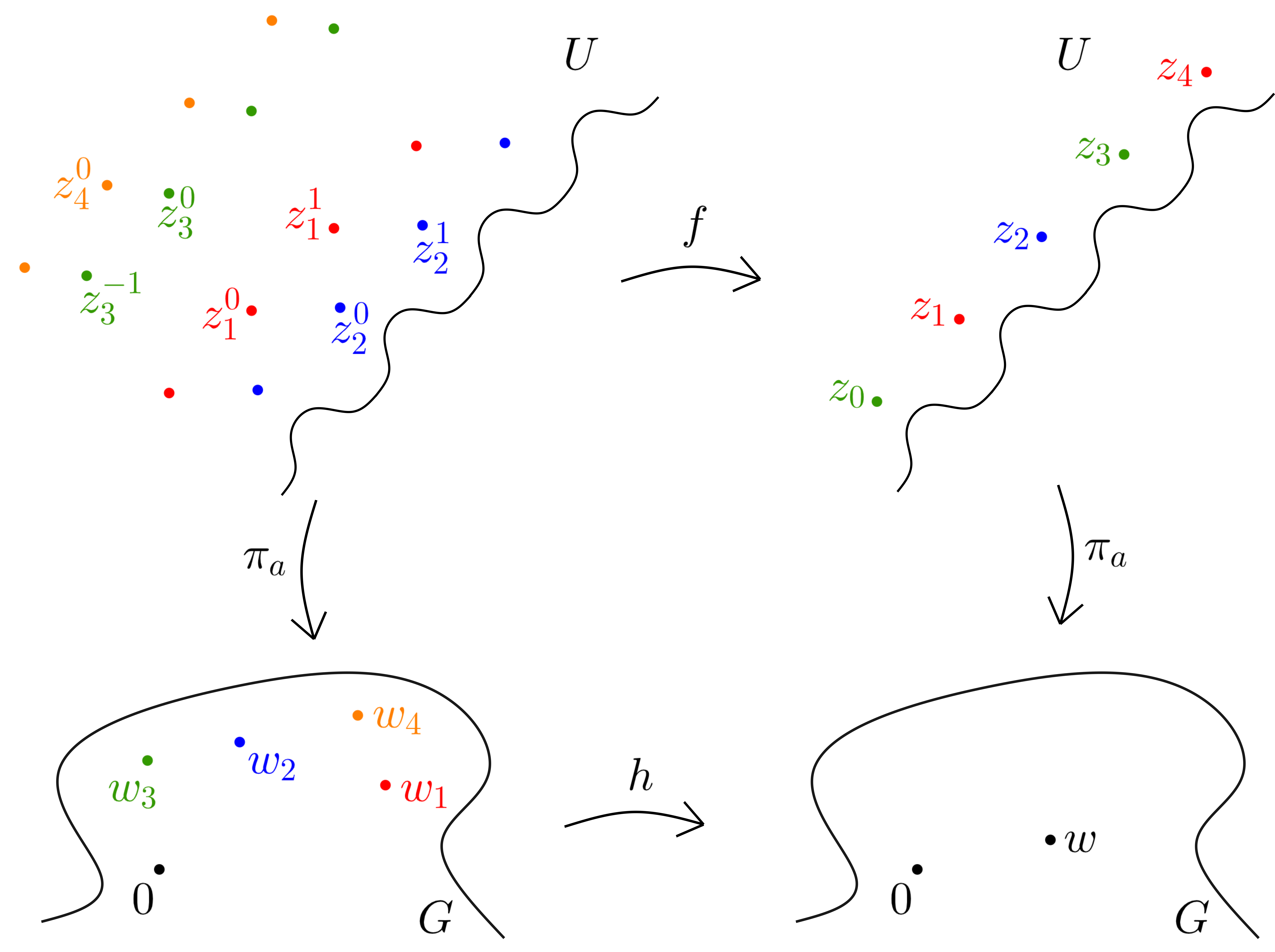}
			\caption{Showing that $f$ has infinite degree in $U$ ($m=3$ here)}
			\label{hinfif}
		\end{figure}
		Now, let $w\in G\setminus\{0\}.$ By the hypothesis, $w$ has infinitely many preimages under $h$ in $G\setminus\{0\},$ say $\{w_k\}_{k\in\N}.$ Let $\{z_k^j\}_{j\in\Z}$ be the preimages of $w_k$ under $\pi_a$ in $U,$ with $z_k^{j+1}=z_k^j+\frac{2\pi i}{a},$ for each $k\in\N$ and $j\in\Z,$ and let $\{z_l\}_{l\in\Z}$ be the preimages of $w$ under $\pi_a$ in $U,$ with $z_{l+1}=z_l+\frac{2\pi i}{a},$ for all $l\in\Z$ (see Figure \ref{hinfif}). By \eqref{berl}, we deduce that $f(z_k^j)\in\{z_l: l\in\Z\},$ for all $k\in\N$ and $j\in\Z.$ Without loss of generality, we may assume that $f(z_1^0)=z_1.$ It follows from \eqref{proli} that $f(z_1^j)=z_{1+jm},$ for all $j\in\Z.$ Since $f$ is univalent in $U,\,f(z_2^0)\notin\{z_{1+jm}: j\in\Z\},$ so we may assume without loss of generality that $f(z_2^0)=z_2.$ Then, using again \eqref{proli}, we get that $f(z_2^j)=z_{2+jm},$ for all $j\in\Z.$ By continuing this way, we obtain that $f(\{z_k^j: j\in\Z,\,k=1,\dots,m\})=\{z_l: l\in\Z\}$ (see Figure \ref{hinfif}). But $f(z_{m+1}^0)\in\{z_l: l\in\Z\},$ so $f$ is not univalent in $U$; a contradiction. Thus, $f|_U$ has infinite degree.
			
		\begin{figure}[h]
			\centering
			\includegraphics[width=0.6\linewidth]{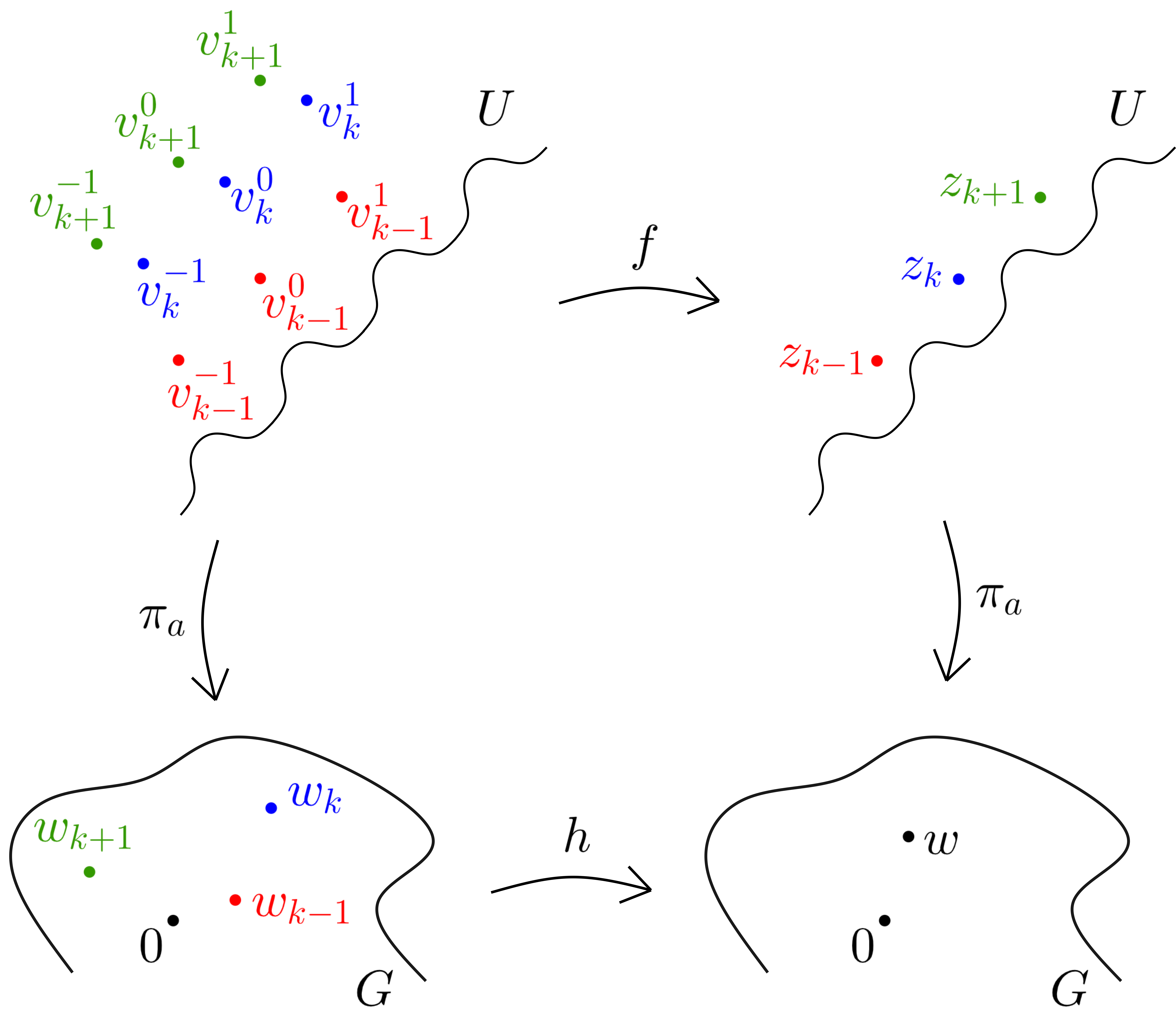}
			\caption{Showing that $h$ has infinite degree in $G$}
			\label{finfih}
		\end{figure}
		For the opposite direction, assume that $f|_U$ has infinite degree. Let $w\in G\setminus\{0\}$ and $\{z_k\}_{k\in\Z}$ be its preimages under $\pi_a $ in $U,$ with $z_{k+1}=z_k+\frac{2\pi i}{a},$ for all $k\in\Z.$ Also, let $\{z_{k,\,n}\}_{n\in\N}$ be the preimages of $z_k$ under $f$ in $U,$ for each $k\in\Z.$ Fix $k\in\Z.$ If there exist $n,\,n'\in\N$ such that 
		\begin{equation} \label{prj}
			z_{k,\,n'}=z_{k,\,n}+\frac{2l\pi i}{a},
		\end{equation}
		for some $l\in\Z\setminus\{0\},$ then \eqref{proli} implies that $m=0$ and $f|_U$ is $\frac{2\pi i}{a}$-periodic. Thus, for each $k\in\Z,$ the set of preimages of $z_k$ under $f$ in $U,\,\{z_{k,\,n}\}_{n\in\N},$ contains a subset $\{v_k^j\}_{j\in\Z},$ such that $v_k^{j+1}=v_k^j+\frac{2\pi i}{a},$ for all $j\in\Z,$ and so $\pi_a(v_k^j)=:w_k\in G,$ for all $j\in\Z$ (see Figure \ref{finfih}). By \eqref{berl}, the points $w_k,\,k\in\Z,$ are preimages of $w$ under $h$ in $G,$ and they are infinitely many, so $h|_G$ has infinite degree.	If, for $k\in\Z$ fixed, there do not exist $n,\,n'\in\N$ with the property \eqref{prj}, then the points $\pi_a(z_{k,\,n})\in G,\,n\in\N,$ are infinitely many, and they are preimages of $w$ under $h$ in $G,$ hence $h|_G$ has infinite degree.
		
		Finally, assume that $h|_G$ has finite degree. By what we have shown above, $f|_U$ has finite degree too, so Lemma \ref{fdegf} implies that $m\ge1$ and $\deg f|_U=1.$ Also, since $h|_G$ is a proper map (see \cite[Proposition 2.8]{BerFR15}) and $h^{-1}(\{0\})=\{0\},$ with the local degree of $h$ at 0 being equal to $m,$ we have that $\deg h|_G=m.$
	\end{proof}
	\begin{rmk}
		Under the hypotheses of Proposition \ref{degrees0}, with $\deg h|_G<+\infty,$ one can observe that $h$ does not have critical points in $G,$ except possibly for 0. Indeed, if $w\ne0$ is a critical point of $h$ in $G,$ then $\pi_a^{-1}(\{w\})\subset U$ will consist of infinitely many critical points of $f,$ because of \eqref{berl}; a contradiction to the fact that $f$ is univalent in $U.$
	\end{rmk}
	\noindent The case where $0\notin G$ is quite simpler, as one can observe in the following proposition.
	\begin{prop} \label{degreesn0}
		Let $h\in\mathcal{L},\,f\in\mathcal{L}_{a,\,h},\,a\in\C^*,$ and $G$ be a forward-invariant Fatou component of $h,$ with $0\notin G.$ Then $\textnormal{deg}\,h|_G=\textnormal{deg}\,f|_{U_k},$ for all $k\in\Z,$ where $\{U_k\}_{k\in\Z}$ are the connected components of $\pi_a^{-1}(G).$
	\end{prop}
	\begin{proof}
		The proof follows directly from the fact that $\pi_a$ is conformal in each $U_k$ (see comment in Remark \ref{scu}(1)).
	\end{proof}
	We finish this section with the proof of Proposition \ref{genrem}, which describes the different types and degrees of Baker domains we obtain when we lift a Fatou component $G,$ which contains the origin. For the classification of Baker domains to doubly-parabolic, simply-parabolic and hyperbolic, see, for example, \cite[Theorem 2.1]{Jov24}.
	\begin{proof}[Proof of Proposition \ref{genrem}]
		Since $m\ge1,$ 0 is a fixed point of $h.$ Note that if $m=1$ and Re$H(0)>0,$ then $|h'(0)|=|e^{H(0)}|=e^{\text{Re}H(0)}>1,$ so 0 is a repelling fixed point of $h,$ and hence $0\in J(h);$ a contradiction. Thus, if $m=1,$ then either Re$H(0)=0,$ or Re$H(0)<0.$
		
		Suppose that $m=1$ and Re$H(0)=0.$ Then $|h'(0)|=1,$ and, since we have assumed that $0\in G\subset F(h),$ we deduce that 0 is an irrationally indifferent fixed point of $h$ and $G$ is a Siegel disc, hence $\deg h|_G=m=1.$ It is easy to see that $U=\pi_a^{-1}(G),$ which is forward-invariant by Remark \ref{scu}(1), cannot be an attracting basin, a parabolic basin, or a Siegel disc, so it is a Baker domain of $f.$ Working as in the proof of \cite[Lemma 3.3]{barg}, we can show that $f|_U\sim\text{id}_\mathbb{H}\pm1$ and, more precisely, $\text{id}_\mathbb{H}\pm1$ is an inner function dynamically associated to $f|_U$. Thus, $f|_U$ is univalent and $U$ is a simply-parabolic Baker domain.
		
		Next, let $m=1$ and Re$H(0)<0.$ Then $0<|h'(0)|<1,$ so 0 is an attracting fixed point of $h,$ and hence $G$ is the immediate attracting basin of 0. By \eqref{berl}, $h^n(e^{az})=e^{af^n(z)},$ for all $z\in\C$ and $n\in\N,$ and since $h^n(w)\to0$ as $n\to\infty,$ for all $w\in G,$ we deduce that $f^n(z)\to\infty$ as $n\to\infty,$ for all $z\in U.$ This shows that $U$ is a Baker domain of $f.$ Working as in the proof of \cite[Lemma 3.3]{barg}, we can show that $f|_U\sim\textnormal{id}_\C+1,$ which means that $U$ is a doubly-parabolic Baker domain. Doubly-parabolic Baker domains are non-univalent, so Proposition \ref{degrees0} implies that $h$ has infinite degree in $G,$ and so $f|_U$ has infinite degree.
		
		Finally, let $m\ge2.$ Then 0 is a super-attracting fixed point of $h$ and $G$ is its super-attracting basin. It follows from \eqref{berl} that $U$ is a Baker domain of $f.$ Working as in the proof of \cite[Lemma 3.3]{barg}, we can show that there exists $\lambda>1$ such that $f|_U\sim\lambda\cdot\text{id}_\mathbb{H},$ which means that $U$ is a hyperbolic Baker domain. If $\deg h|_G<+\infty,$ then it follows from Proposition \ref{degrees0} that $\deg h|_G=m\ge2$ and $\deg f|_U=1.$ Otherwise, $f|_U$ has infinite degree.
	\end{proof}
	\section{Proof of Proposition \ref{wandinn}} \label{wdinn}
	We start by proving Proposition \ref{wandinn}, which is an independent result. It can be applied to lifts of transcendental entire functions, and it will be used in the proof of Theorem \ref{const}.
	\begin{proof}[Proof of Proposition \ref{wandinn}]
		For (i), we can show by induction that 
		\begin{equation} \label{fprop}
			f(z+kc)=f(z)+mkc,
		\end{equation} 
		for all $k\in\Z$ and $z\in\C.$ Recall that $\hat{f}(z)=f(z)+lc,$ where $l\in\Z.$ Using \eqref{fprop}, we can show, again by induction, that 
		$$\hat{f}^n(z)=f^n(z)+(1+m+\dots+m^{n-1})lc,$$ for all $n\in\N$ and $z\in\C,$ which implies that $\hat{f}^n(z)-\hat{f}^n(w)=f^n(z)-f^n(w),$ for all $n\in\N$ and $z,\,w\in\C.$ 
		
		Let $z_0\in F(\hat{f}).$ Then there exists a neighbourhood $N(z_0)$ of $z_0$ in which $\{\hat{f}^n\}_{n\in\N}$ is equicontinuous. Thus, if $w\in N(z_0)$ and $\varepsilon>0,$ there exists $\delta=\delta(w,\,\varepsilon,\,\hat{f})>0$ such that for all $z\in\C$ and $n\in\N,$ $$|z-w|<\delta \;\text{ implies }\; |f^n(z)-f^n(w)|=|\hat{f}^n(z)-\hat{f}^n(w)|<\varepsilon.$$ It follows that the family $\{f^n\}_{n\in\N}$ is equicontinuous in $N(z_0),$ hence $z_0\in F(f).$ Similarly, we can show that $F(f)\subset F(\hat{f}),$ and so $F(\hat{f})=F(f).$
		
		For (ii), note first that 
		\begin{equation} \label{imuk}
			f(U_k)\subset U_{k_0+mk}, \text{ for all } k\in\Z,
		\end{equation}
		by \eqref{fprop} and hypotheses (a) and (b). Next, consider the sets $$\mathcal{R}_k:=\{\varphi_k:\D\to U_k\,|\,\varphi_k \text{ is a Riemann map}\},\,k\in\Z,$$ for which one can observe that, for any $\kappa,\,\lambda\in\Z,$ 
		\begin{equation} \label{kimpl}
			\varphi_\kappa\in\mathcal{R}_\kappa \;\Rightarrow\; \varphi_\kappa+(\lambda-\kappa)c\in\mathcal{R}_\lambda, 
		\end{equation}
		and
		\begin{equation} \label{invphi}
			\varphi_\kappa\in\mathcal{R}_\kappa \;\Rightarrow\; \varphi_\kappa^{-1}(w)=(\varphi_{\kappa}+(\lambda-\kappa)c)^{-1}(w+(\lambda-\kappa)c), \text{ for all } w\in U_\kappa.
		\end{equation}
		
		Now, fix $j,\,j'\in\Z.$ By \eqref{imuk}, the form of $\hat{f}$ and hypothesis (a) we have that $f(U_j)\subset U_{k_0+mj}$ and $\hat{f}(U_{j'})\subset U_{k_0+mj'+l}.$ Next, let $\varphi_k\in\mathcal{R}_k,$ for $k\in\{j,\,k_0+mj,\,j',\,k_0+mj'+l\},$ and consider the functions $g_{f,\,j}=\varphi_{k_0+mj}^{-1}\circ f\circ\varphi_j$ and $g_{\hat{f},\,j'}=\varphi_{k_0+mj'+l}^{-1}\circ \hat{f}\circ\varphi_{j'},$ which are inner functions associated to $f|_{U_j}$ and $\hat{f}|_{U_{j'}},$ respectively. Then, for all $z\in\D,$ and for $\kappa=k_0+mj,\,\lambda=k_0+mj'+l$ and $w=\hat{f}(\varphi_j(z)+(j'-j)c)-m(j'-j)c-lc$, \eqref{invphi} gives
		\begin{equation} \nonumber
			\begin{split}
				g_{f,\,j}(z)&=\varphi_{k_0+mj}^{-1}\circ f\circ\varphi_j(z)=\varphi_{k_0+mj}^{-1}(f(\varphi_j(z)+(j'-j)c-(j'-j)c))\\&\overset{\eqref{fprop}}{=}\varphi_{k_0+mj}^{-1}(f(\varphi_j(z)+(j'-j)c)-m(j'-j)c)\\&=\varphi_{k_0+mj}^{-1}(\hat{f}(\varphi_j(z)+(j'-j)c)-m(j'-j)c-lc)\\&\overset{\eqref{invphi}}{=}(\varphi_{k_0+mj}+(k_0+mj'+l-k_0-mj)c)^{-1}(\hat{f}(\varphi_j(z)+(j'-j)c))\\&=(\varphi_{k_0+mj}+(k_0+mj'+l-k_0-mj)c)^{-1}\circ \hat{f}\circ(\varphi_j+(j'-j)c)(z),
			\end{split}
		\end{equation} 
		and \eqref{kimpl} implies that $\varphi_j+(j'-j)c\in\mathcal{R}_{j'}$ and $\varphi_{k_0+mj}+(k_0+mj'+l-k_0-mj)c\in\mathcal{R}_{k_0+mj'+l},$ so $g_{f,\,j}$ is associated to $\hat{f}|_{U_{j'}}.$ Similarly, for $z\in\D,$ and for $\kappa=k_0+mj'+l,\,\lambda=k_0+mj$ and $w=f(\varphi_{j'}(z)+(j-j')c)-m(j-j')c+lc,$ \eqref{invphi} gives
		\begin{equation} \nonumber
			\begin{split}
				g_{\hat{f},\,j'}(z)&=\varphi_{k_0+mj'+l}^{-1}\circ \hat{f}\circ\varphi_{j'}(z)=\varphi_{k_0+mj'+l}^{-1}(\hat{f}(\varphi_{j'}(z)+(j-j')c-(j-j')c))\\&=\varphi_{k_0+mj'+l}^{-1}(f(\varphi_{j'}(z)+(j-j')c-(j-j')c)+lc)\\&\overset{\eqref{fprop}}{=}\varphi_{k_0+mj'+l}^{-1}(f(\varphi_{j'}(z)+(j-j')c)-m(j-j')c+lc)\\&\overset{\eqref{invphi}}{=}(\varphi_{k_0+mj'+l}+(k_0+mj-k_0-mj'-l)c)^{-1}(f(\varphi_{j'}(z)+(j-j')c))\\&=(\varphi_{k_0+mj'+l}+(k_0+mj-k_0-mj'-l)c)^{-1}\circ f\circ(\varphi_{j'}+(j-j')c)(z),
			\end{split}
		\end{equation}
		and \eqref{kimpl} implies that $\varphi_{j'}+(j-j')c\in\mathcal{R}_j$ and $\varphi_{k_0+mj'+l}+(k_0+mj-k_0-mj'-l)c\in\mathcal{R}_{k_0+mj},$ so $g_{\hat{f},\,j'}$ is associated to $f|_{U_j}.$
	\end{proof}
	\begin{rmk} \label{wdpli}
		\begin{enumerate}[(1), leftmargin=*, itemsep=0pt, topsep=0pt]
			\item Let $f$ be as in Proposition \ref{wandinn}. 
			\begin{enumerate}
				\item[(a)] If $m=0,$ then \eqref{imuk} implies that $f(U_{k})\subset U_{k_0},$ for all $k\in\Z,$ so $U_{k_0}$ is forward-invariant and $U_k$ is a preimage component of the forward-invariant Fatou component $U_{k_0},$ for each $k\in\Z.$
				\item[(b)] If $m=1,$ then either $k_0=0,$ and so $U_k$ is forward-invariant for all $k\in\Z$ by \eqref{imuk}, or $k_0\ne0.$ In the latter case, it follows from \eqref{imuk} that $f(U_{nk_0+k})\subset U_{(n+1)k_0+k},$ for all $n\in\N\cup\{0\}$ and $k\in\Z,$ hence $U_k$ is a wandering domain for each $k\in\Z.$ 
				\item[(c)] Finally, if $m\ge2,$ then set $K:=-\frac{k_0}{m-1}$ and note that $k_0+mK=K.$ For each $k\in\Z,$ consider the sequence $\{\alpha_{k,\,n}\}_{n=0}^{+\infty}\subset\Z,$ defined by $$\alpha_{k,\,0}=k \text{ and } \alpha_{k,\,n}=k_0(1+m+\dots+m^{n-1})+m^nk,\,n\ge1.$$ It follows from \eqref{imuk} that $f(U_{\alpha_{k,\,n}})\subset U_{\alpha_{k,\,n+1}},$ for all $n\ge0.$ Also, if $k>K,$ then $\{\alpha_{k,\,n}\}_{n=0}^{+\infty}$ is strictly increasing, whereas if $k<K,\,\{\alpha_{k,\,n}\}_{n=0}^{+\infty}$ is strictly decreasing. Now, if $K\in\Z,$ then $f(U_K)\subset U_{k_0+mK}=U_K$ by \eqref{imuk}, so $U_K$ is forward-invariant, and from the above discussion we have that $U_k$ is a wandering domain for each $k\in\Z\setminus\{K\}.$ If $K\notin\Z,\,U_k$ is a wandering domain for each $k\in\Z.$
			\end{enumerate}
			The same conclusions can be derived for a function $\hat{f}$ defined as in Proposition \ref{wandinn} (note that such a function $\hat{f}$ has the same property with $f$: $\hat{f}(z\pm c)=\hat{f}(z)\pm mc$).
			\item Let $h$ be a function of the form \eqref{genf}, and $f,\,\hat{f}\in\mathcal{L}_{a,\,h},\,a\in\C^*.$ Also, assume that $h$ has a forward-invariant Fatou component $G$ such that $0\notin G,$ and let $\{U_k\}_{k\in\Z}$ be the connected components of $\pi_a^{-1}(G).$ Then \eqref{berl}-\eqref{proli}, Corollary \ref{liftgl}(ii) and Remark \ref{scu}(1) imply that $f,\,\hat{f}$ satisfy the hypotheses of Proposition \ref{wandinn} for $c=\frac{2\pi i}{a}.$ Thus, it follows from Proposition \ref{wandinn}(ii) that, if $k\in\Z$ and $g_{f,\,k},\,g_{\hat{f},\,k}$ are inner functions associated to $f|_{U_k},\,\hat{f}|_{U_k},$ respectively, then $g_{f,\,k}$ is associated to $\hat{f}|_{U_k}$ and $g_{\hat{f},\,k}$ is associated to $f|_{U_k}.$ 
		\end{enumerate}	 
	\end{rmk}
	\section{Infinite-degree case (Proof of Theorem \ref{const})} \label{inf}
	In this section we prove Theorem \ref{const} and Corollary \ref{infpoly}. In the proof we use \cite[Theorem 6.1]{ERS20}, which we state below for completeness. 
	\begin{thm}[{\protect\cite[Theorem~6.1]{ERS20}}] \label{ers}
		Suppose that $f$ is an entire function and $U$ is an unbounded forward-invariant Fatou component on which $f$ has infinite valence, but such that $f^{-1}(a)\cap U$ contains exactly $p$ points, counting multiplicity, for some $a\in U$ and $p\ge 0.$ Assume also that an inner function dynamically associated with $f|_U$ has a finite number $q\ge1$ of singularities on $\partial\D.$ Then $f$ has a dynamically associated inner function of the form $$g:\D\to\D;\;\;\; z\mapsto B(z)\exp\bigg(-\sum_{j=1}^{q}\Big(c_j\frac{e^{i\theta_j}+z}{e^{i\theta_j}-z}\Big)\bigg),$$ for some finite Blaschke product $B$ of degree $p,$ real numbers $\theta_1,\dots,\,\theta_q,$ and positive real numbers $c_1,\dots,\,c_q.$
	\end{thm}
	\begin{proof}[Proof of Theorem \ref{const}]
		The function $h$ satisfies the assumptions of Theorem \ref{ers}, with $b$ instead of $a,$ so $h|_G$ has a dynamically associated inner function of the form $$g_h(w)=B(w)\exp\Big(-\sum_{j=1}^{q}\Big(c_j\cdot\frac{e^{i\theta_j}+w}{e^{i\theta_j}-w}\Big)\Big),\,w\in\D,$$ where $B$ is a finite Blaschke product of degree $p,$ $\theta_j\in\R$ and $c_j>0,$ for all $j\in\{1,\dots,q\}.$ 
		
		First we show (i), so we assume that $0\in G$ (note that, in this case, $U=\pi_a^{-1}(G)$ is connected by Corollary \ref{liftgl}(i)). If $m=0,$ then $h(w)=e^{H(w)},$ and so 0 has no preimages in $G.$ Since $h|_G$ has infinite degree, there exists at most one point in $G$ with finitely many preimages in $G$ under $h,$ hence $b=0,\,p=0$ and $B(w)=e^{i\theta},$ for all $w\in\D,$ and some $\theta\in[0,2\pi).$ If $m\ge1,$ then $h^{-1}(\{0\})=\{0\}$ and $h$ has local degree $m$ at 0, hence we have again that $b=0,$ and so $p=m.$ Also, if $\psi:\D\to G$ is the Riemann map such that $g_h=\psi^{-1}\circ h\circ\psi:\D\to\D,$ then $\psi(0)=b=0$ (see the proof of Theorem \ref{ers}). Thus, $$g_h(w)=0 \;\Leftrightarrow\; h(\psi(w))=\psi(0)=0 \;\Leftrightarrow\; \psi(w)=0 \;\Leftrightarrow\; w=0,$$ so 0 is the only root of $B.$ This means that $B(w)=e^{i\theta}w^m,$ for all $w\in\D,$ and some $\theta\in[0,2\pi).$ 
		
		Now, recall that $U=\pi_a^{-1}(G)$ is a forward-invariant Fatou component of $f$ (see Remark \ref{scu}(1)). Let $\varphi:\mathbb{H}\to U$ be a Riemann map and consider the inner function $\tilde{g}_f=\varphi^{-1}\circ f\circ\varphi:\mathbb{H}\to\mathbb{H},$ which is dynamically associated to $f|_U.$ Consider, also, the function $\tilde{p}=\psi^{-1}\circ\pi_a\circ\varphi:\mathbb{H}\to\mathbb{D}^*,$ where $\mathbb{D}^*=\mathbb{D}\setminus\{0\}.$ Then $\tilde{p}$ is a universal covering map from $\mathbb{H}$ onto $\mathbb{D}^*,$ and $$g_h\circ\tilde{p}=g_h\circ\psi^{-1}\circ\pi_a\circ\varphi=\psi^{-1}\circ h\circ\pi_a\circ\varphi=\psi^{-1}\circ\pi_a\circ f\circ\varphi=\psi^{-1}\circ\pi_a\circ\varphi\circ\tilde{g}_f=\tilde{p}\circ\tilde{g}_f,$$ in $\mathbb{H}.$ Our goal is to calculate $\tilde{g}_f,$ but since we do not know $\tilde{p},$ we will calculate another inner function dynamically associated to $f|_U.$ Note that  $\hat{p}(z)=e^{2iz}$ is a universal covering map from $\mathbb{H}$ onto $\mathbb{D}^*,$ and recall that the universal covering map is unique up to isomorphism (see, for example, \cite[Remark 12.18]{zakeri}), which means that there exists a conformal automorphism $M:\mathbb{H}\to\mathbb{H}$ such that $\hat{p}=\tilde{p}\circ M.$ Then, it suffices to find the function $g_f=M^{-1}\circ\tilde{g}_f\circ M:\mathbb{H}\to\mathbb{H},$ which is conjugate to $\tilde{g}_f.$ See the diagram in Figure \ref{diagr} for a visual description of the above process. 
		\begin{figure}[h]
			\centering
			$$\begin{tikzcd} 
				\mathbb{H} \arrow[r, "g_f"] \arrow[d, "M"'] \arrow[to=dddd, bend right=60, "\hat{p}"'] & \mathbb{H} \\
				\mathbb{H} \arrow[r, "\tilde{g}_f"] \arrow[d, "\varphi"'] \arrow[to=ddd, bend right=40, "\tilde{p}"'] & \mathbb{H} \arrow[u, "M^{-1}"'] \\
				U \arrow[r, "f"] \arrow[d, "\pi_a"'] & U \arrow[d, "\pi_a"] \arrow[u, "\varphi^{-1}"'] \\
				G \arrow[r, "h"] & G \arrow[d, "\psi^{-1}"] \\
				\D \arrow[r, "g_h"] \arrow[u, "\psi"] & \D
			\end{tikzcd}$$
			\caption{Constructing the dynamically associated inner function $g_f$ to $f|_U$}
			\label{diagr}
		\end{figure}
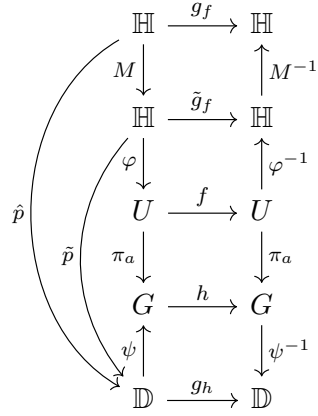	
		
		Let $z\in\mathbb{H}.$ Then there exists a point $w\in\mathbb{H}$ such that $z=M^{-1}(w).$ Thus, we have that 
		\begin{equation} \nonumber
			\begin{split}
				g_h(\hat{p}(z))&=g_h(\hat{p}(M^{-1}(w)))=g_h(\tilde{p}(w))=\tilde{p}(\tilde{g}_f(w))=\tilde{p}(\tilde{g}_f(M(z)))\\&=\tilde{p}(M(g_f(z)))=\hat{p}(g_f(z)),
			\end{split}
		\end{equation}
		so 
		\begin{equation} \nonumber
			\begin{split}
				&g_h(e^{2iz})=e^{2ig_f(z)} \;\Leftrightarrow\; e^{i\theta}e^{2miz}\cdot\exp\Big(-\sum_{j=1}^{q}\Big(c_j\cdot\frac{e^{i\theta_j}+e^{2iz}}{e^{i\theta_j}-e^{2iz}}\Big)\Big)=e^{2ig_f(z)}\\ \Leftrightarrow\;\;&2ig_f(z)=i\theta+2miz-\sum_{j=1}^{q}\Big(c_j\cdot\frac{e^{i\theta_j}+e^{2iz}}{e^{i\theta_j}-e^{2iz}}\Big)+2k_z\pi i,\,k_z\in\Z \\ \Leftrightarrow\;\; &g_f(z)=\frac{\theta}{2}+k_z\pi+mz+\frac{i}{2}\cdot\sum_{j=1}^{q}\Big(c_j\cdot\frac{e^{i\theta_j}+e^{2iz}}{e^{i\theta_j}-e^{2iz}}\Big),\,k_z\in\Z,
			\end{split}		
		\end{equation}
		where $k_z$ depends on $z$ and takes its values in $\Z.$ Since $k_z$ is also a continuous function of $z,$ it has to be constant, so $$g_f(z)=\sigma+mz+\frac{i}{2}\cdot\sum_{j=1}^{q}\Big(c_j\cdot\frac{e^{i\theta_j}+e^{2iz}}{e^{i\theta_j}-e^{2iz}}\Big),\,z\in\mathbb{H},$$ where $\sigma=\frac{\theta}{2}+k\pi\in\R,$ for some $k\in\Z.$
		
		Regarding (ii), let $0\notin G$ and $\{U_k\}_{k\in\Z}$ be the components of $\pi_a^{-1}(G)$ (see Corollary \ref{liftgl}(ii)). Recall that $\pi_a$ is conformal in each $U_k$ (see comment in Remark \ref{scu}(1)), and fix $k\in\Z.$ By Remark \ref{scu}(2), there exists a function $\hat{f}\in\mathcal{L}_{a,\,h}$ such that $\hat{f}(U_k)\subset U_k.$ Let $\varphi:\D\to U_k$ be a Riemann map and consider the inner function $g_{\hat{f},\,k}=\varphi^{-1}\circ\hat{f}\circ\varphi:\D\to\D,$ which is dynamically associated to $\hat{f}|_{U_k}.$ Also, let $\tilde{p}:=\psi^{-1}\circ\pi_a\circ\varphi:\D\to\D,$ and note that $\tilde{p}$ is a conformal automorphism of the unit disc. Working as in the proof of (i), we can show that $g_h\circ\tilde{p}=\tilde{p}\circ g_{\hat{f},\,k},$ so $g_h=\tilde{p}\circ g_{\hat{f},\,k}\circ\tilde{p}^{-1}=\psi^{-1}\circ\pi_a\circ\hat{f}\circ\pi_a^{-1}\circ\psi$ in $\D.$ This means that $g_h$ is an inner function dynamically associated to $\hat{f}|_{U_k},$ and hence, $g_h$ is associated to $f|_{U_k}$ by Remark \ref{wdpli}(2). 
	\end{proof}
	\begin{proof}[Proof of Corollary \ref{infpoly}]
		Since 0 has $m$ preimages in $G,$ counting multiplicity, it suffices to show that hypothesis (c) of Theorem \ref{const} holds. We can deduce from the form of $h$ that it has finitely many critical points, while \cite[Proposition 2]{K89} implies that 0 is the only finite asymptotic value of $h,$ so the set of singular values $S(h)$ of $h$ is finite. By doing some calculations, one can show that the order of $h$ is $\rho(h)=\deg Q.$ Since $h$ has finitely many critical points, it follows from \cite[Theorem 1]{BE95} that every singularity over 0 is direct. By the Denjoy-Carleman-Ahlfors Theorem (see, for example, \cite[p. 357]{BE95}), $h$ has at most $2\cdot\deg Q$ direct singularities. For transcendental entire functions, $\infty$ is a direct singularity, and the number of singularities over $\infty$ is at least the number of singularities over points in $\C$ (see \cite[p. 6]{E21}). From this, we conclude that there exist at most $\deg Q$ singularities over 0, and hence, there are at most $\deg Q$ asymptotic paths in $\C$ (up to homotopy) on which $h$ converges to 0. By \cite[Theorem 1.2]{ERS20}, a dynamically associated inner function to $h|_G$ has $1\le q\le\deg Q$ singularities, so hypotheses (a)-(c) of Theorem \ref{const} are satisfied. Since $U=\pi_a^{-1}(G)$ is connected due to Corollary \ref{liftgl}(i), the result follows.
	\end{proof}
	\section{Finite-degree case (Proof of Theorem \ref{uds})} \label{fin}
		In this section we give the proofs of Theorem \ref{uds} and Corollary \ref{finpoly}. For the proof of Theorem \ref{uds} we distinguish between the cases $0\in G$ and $0\notin G,$ as usual. 
		\begin{proof}[Proof of Theorem \ref{uds}]
		Suppose $0\in G$ and note that, in this case, $U=\pi_a^{-1}(G)$ is connected by Corollary \ref{liftgl}(i). The result follows from Proposition \ref{degrees0} and \cite[Proposition 1.1]{ERS20}. Recall that $U$ is forward-invariant (see Remark \ref{scu}(1)). 
		
		Now, assume that $0\notin G,$ and let $\{U_k\}_{k\in\Z}$ be the components of $\pi_a^{-1}(G)$ (see Corollary \ref{liftgl}(ii)). Working exactly as in the proof of Theorem \ref{const}(ii), we can show that $g_h$ is associated to $f|_{U_k},$ for each $k\in\Z.$ Note that the Fatou components $\{U_k\}_{k\in\Z}$ can be forward-invariant, eventually forward-invariant, or wandering domains for $f$ (see Remark \ref{wdpli}(1)). Also, $g_h$ is a finite Blaschke product of degree $d$ by \cite[Proposition 1.1]{ERS20}.
	\end{proof}
	\begin{proof}[Proof of Corollary \ref{finpoly}]
		By \cite[Proposition 2]{K89}, 0 is the only finite asymptotic value of $h,$ and it does not belong to $G.$ Also, it follows from the form of $h$ that it has $m+\deg Q-1$ critical points, counting multiplicity. Thus, \cite[Proposition 2.8]{BerFR15} yields that $h|_G$ is a proper map and has finite degree, $d=m+\deg Q.$ Indeed, if $m=0$ and $\deg Q=1,$ then $h$ does not have critical points, so \cite[Corollary 12.34]{zakeri} implies that $h|_G: G\to G$ is a covering map. Then $d=1=m+\deg Q,$ by the Riemann-Hurwitz formula (see, for example, \cite[Theorem 12.47]{zakeri}). If $m+\deg Q-1\ge1,$ then let $c_1,\dots,c_k$ be the critical points of $h,$ with multiplicities $m_1,\dots,m_k,$ respectively, where $k\in\N$ and $m_1+\dots+m_k=m+\deg Q-1$ (the multiplicities are considered with respect to the equation $h'(w)=0$). Since the local degree of the critical point $c_j$ is $m_j+1,$ for $j=1,\dots,k,$ the Riemann-Hurwitz formula implies that $d=m+\deg Q.$ The rest follows from Theorem \ref{uds}(ii).
	\end{proof}
	\section{Examples} \label{exampl}
	We first introduce some notation we are going to use in the examples. In particular, if $f$ is a transcendental entire function, by $S(f)$ we denote the set of its singular values (more details on singular values can be found, for example, in \cite[Section 4.3]{Ber93}, \cite[p. 355-357]{BE95}, \cite{E21}). Then, we can consider two important classes of transcendental entire functions, namely:
	$$\mathcal{S}=\{f \text{ transcendental entire: } S(f) \text{ is finite}\},$$ which was studied by Nevanlinna and others, and $$\mathcal{B}=\{f \text{ transcendental entire: } S(f) \text{ is bounded}\}.$$ The latter class was introduced by Eremenko and Lyubich in \cite{EL92}, where they obtained important results for both of these classes. 
	
	We begin with two examples of the infinite-degree case. In both examples, the functions satisfy the hypotheses of Proposition \ref{genrem}(c), so they have a hyperbolic Baker domain. The first example is a function studied in \cite[p. 69]{RS99} and \cite[Section 4]{RRS10}.
	\begin{eg} \label{ex1}
		Let $f(z)=2z+e^{-z},\,z\in\C,$ and $U$ be the Baker domain of $f.$ Then $g_f(z)=2z-\frac{c}{2}\cdot\cot z,\,z\in\mathbb{H},$ where $c>0,$ is an inner function dynamically associated to $f|_U.$ 
	\end{eg}
	\begin{proof}
		Using \eqref{berl} with $\pi_{-1}(z)=e^{-z},\,z\in\C,$ we find that $f$ is a lift of the function $$h(w)=w^2e^{-w},\,w\in\C.$$  
		
		We will first study the dynamics of $h.$ We have that $$h'(w)=2we^{-w}-w^2e^{-w},\,w\in\C,$$ so 0 is a super-attracting fixed point of $h$ (and it is the only real fixed point of $h$). Let $G\subset F(h)$ be the immediate super-attracting basin of 0. By studying the function $x\mapsto h(x)-x,\,x\in\R,$ we can deduce that $h(x)<x,$ for all $x>0.$ Also, $h(x)>0,$ for all $x\in\R\setminus\{0\}.$ It follows that $h^n(x)\to0$ as $n\to\infty,$ for all $x\in\R,$ and so $\R\subset G.$ 
		
		Next, we have that $$h'(w)=0 \;\Leftrightarrow\; w=0 \;\text{ or }\; w=2,$$ so $h(0)=0$ and $h(2)=4e^{-2}$ are the only critical values of $h$ and they both belong to $G.$ By \cite[Proposition 2]{K89}, 0 is the only finite asymptotic value of $h,$ hence $S(h)=\{0,4e^{-2}\}\subset G$ and $h\in\mathcal{S}\subset\mathcal{B}.$ It follows from \cite[Theorem 1.7]{EFXS19} that $G=F(h)$ and $h$ is of disjoint type. One can further observe that $J(h)$ is a Cantor bouquet, since $h$ is of disjoint type and has order 1 (see \cite[Theorem 1.5]{BJR12}), and that $G$ is symmetric with respect to the real axis, because $h(\bar{w})=\overline{h(w)},$ for all $w\in\C$ (see Figure \ref{eg1h}).
		\begin{figure}[h]
			\centering
			\begin{subfigure}{0.45\textwidth}
				\centering
				\includegraphics[width=\linewidth]{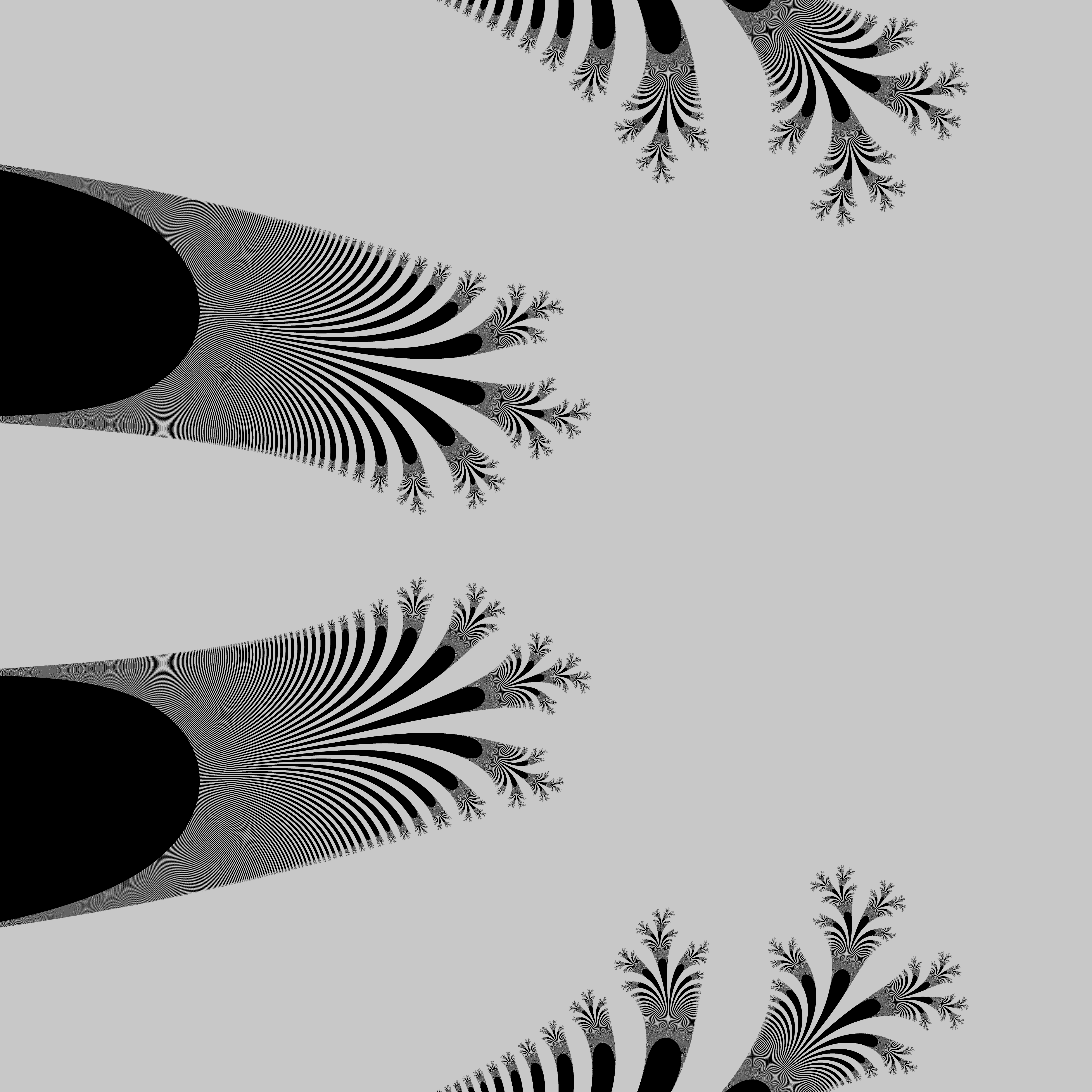}
				\caption{(a) The Fatou (grey) and the Julia set (black) of $h.$}
				\label{eg1h}
			\end{subfigure}
			\hfill
			\begin{subfigure}{0.45\textwidth}
				\centering
				\includegraphics[width=\linewidth]{"julia_set"}
				\caption{(b) The Fatou (grey) and the Julia set (black) of $f.$}
				\label{eg1f}
			\end{subfigure}
			\caption{The dynamical planes of the function $h(w)=w^2e^{-w}$ and of its lift $f(z)=2z+e^{-z}.$}
		\end{figure}	
		
		We will now describe some of the dynamical properties of $f.$ By \eqref{fh}, $U:=\pi_{-1}^{-1}(G)=F(f),$ and the definition of $U$ implies that $U+2\pi i=U.$ Also, Corollary \ref{liftgl}(i) yields that $U$ is connected, so $f|_U$ has infinite degree. It follows from Remark \ref{scu}(1) and Proposition \ref{genrem}(c) that $U$ is a forward-invariant hyperbolic Baker domain, while the fact that $f(\bar{z})=\overline{f(z)},$ for all $z\in\C,$ yields that $U$ is symmetric with respect to the real axis. Note that, by using \eqref{fh}, we can show that $x+k\pi i\in F(f),$ for all $x\in\R$ and $k\in\Z.$ 
		
		Next, if $z=x+iy\in\mathbb{C},$ with $x>0.6,$ then Re$f(z)=2x+e^{-x}\cos y\ge2x-e^{-x}>x.$ Thus, $f(P)\subset P,$ where $P:=\{z\in\C: \textnormal{Re}z>0.6\},$ and so, $P\subset F(f)=U.$ Also, working as in the proof of \cite[Lemma 2.7]{KU}, we can show that there exists $\delta\in(0,\frac{\pi}{3}]$ such that
		\begin{equation} \nonumber
			\begin{aligned}
				J(f)\subset\bigcup_{k\in\Z}\Big(&\{z\in\C: \text{Re}z\le0.6,\,2k\pi+\delta<\text{Im}z<(2k+1)\pi\} \\ &\cup\{z\in\C: \text{Re}z\le0.6,\,(2k+1)\pi<\text{Im}z<2(k+1)\pi-\delta\}\Big),
			\end{aligned}
		\end{equation}
		and $J(f)$ consists of disjoint curves tending to $\infty,$ each homeomorphic to $[0,+\infty),$ by \cite[Theorem 1.3]{RRS10} (see Figure \ref{eg1f}).  
				
		Note that $h|_G$ has infinite degree, and hence it satisfies the hypotheses of Corollary \ref{infpoly}. Thus, a dynamically associated inner function to $f|_U$ is 
		\begin{equation} \label{ex1g}
			\hat{g}_f(z)=\sigma+2z+\frac{ic}{2}\cdot\frac{e^{i\theta}+e^{2iz}}{e^{i\theta}-e^{2iz}},\,z\in\mathbb{H},
		\end{equation}
		where $\sigma,\,\theta\in\R$ and $c>0.$ 
		
		Next, we calculate $\sigma$ and $\theta.$ In order to do this, we make some important observations regarding the proof of Theorem \ref{const}(i), following an idea in the proof of \cite[Theorem 1.9]{ERS20}. In particular, we show that, for an appropriate choice of the Riemann map $\varphi:\mathbb{H}\to U,$ used in the proof of Theorem \ref{const}(i),  
		\begin{equation} \label{varphi}
			\varphi^{-1}(z)-\varphi^{-1}(z+2\pi i)=\pi,\text{ for all } z\in U.
		\end{equation}
		
		Let $\alpha>0.$ Without loss of generality, we may assume that $\varphi(i\alpha)=0$ and $i\varphi'(i\alpha)>0.$ Also, for any $\xi>0$ and $k\in\Z,$ consider the geodesic rays
		\begin{equation} \label{vertr}
			l_{k,\,\xi}^+:=\{z\in\C: \text{Re}z=k\pi,\,\text{Im}z>\xi\} \text{ and } l_{k,\,\xi}^-:=\{z\in\C: \text{Re}z=k\pi,\,0<\text{Im}z<\xi\},
		\end{equation}
		and the hyperbolic geodesics 
		\begin{equation} \label{vertgeo}
			l_k:=\{z\in\mathbb{H}: \text{Re}z=k\pi\}=l_{k,\,\xi}^+\cup l_{k,\,\xi}^-\cup\{k\pi+i\xi\}
		\end{equation}
	    in $\mathbb{H},$ as well as the rays $$\eta_k^+:=\{z\in\C: \text{Re}z>0,\,\text{Im}z=2k\pi\} \text{ and } \eta_k^-:=\{z\in\C: \text{Re}z<0,\,\text{Im}z=2k\pi\},$$ and the horizontal lines $$\eta_k:=\{z\in\C: \text{Im}z=2k\pi\}=\eta_k^+\cup \eta_k^-\cup\{2k\pi i\}$$ in $U.$ By the symmetry of $U$ with respect to the real axis, we have that 	
		\begin{equation} \label{phi}
			\varphi(l_{0,\,\alpha}^+)=(0,+\infty)=\eta_0^+ \text{ and } \varphi(l_{0,\,\alpha}^-)=(-\infty,0)=\eta_0^-
		\end{equation}
		(see Figure \ref{mhyp}). 
		\begin{figure}[h]
			\centering
			\includegraphics[width=0.85\linewidth]{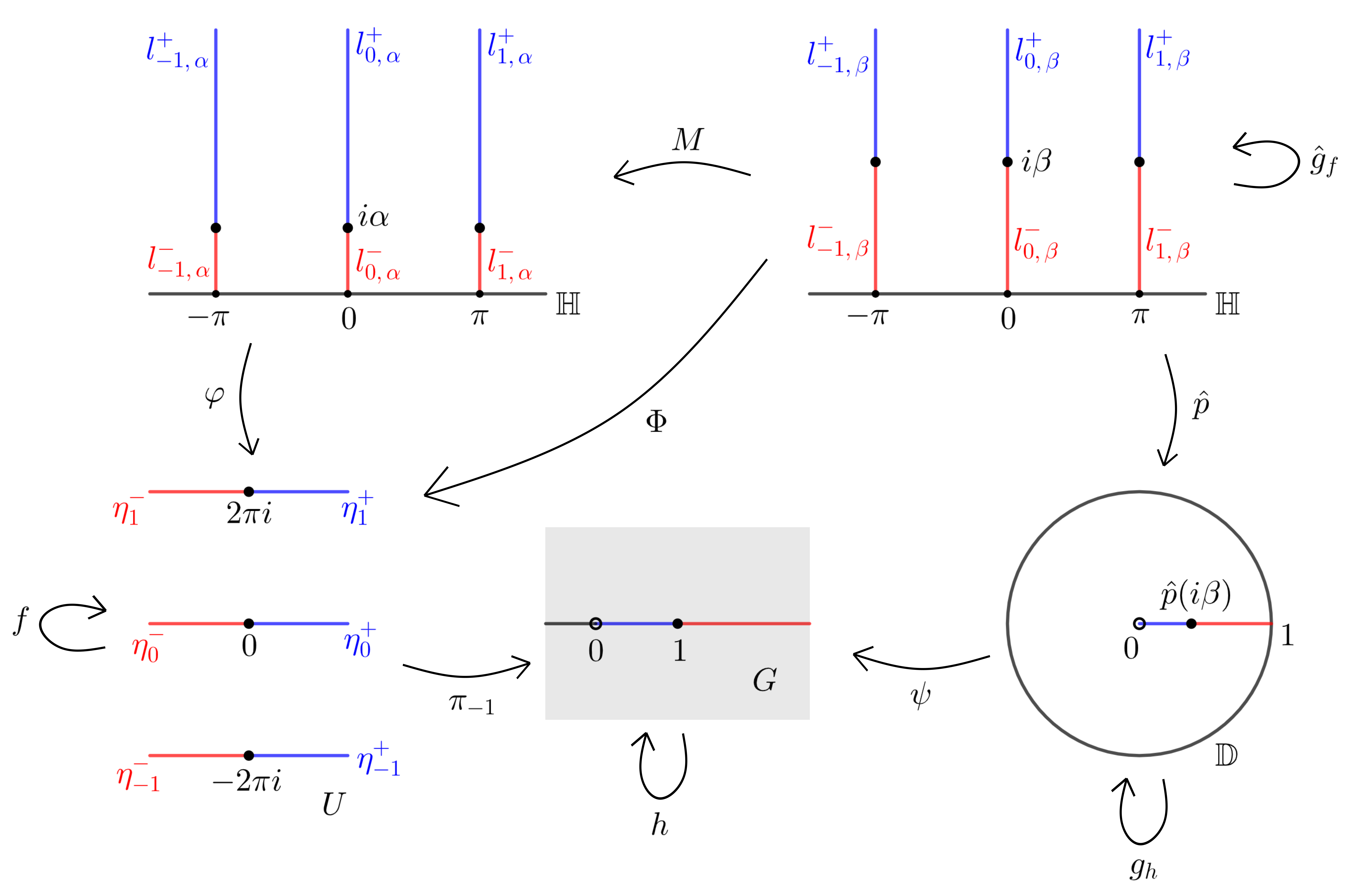}
			\caption{}
			\label{mhyp}
		\end{figure}
		Since $l_0$ is a hyperbolic geodesic in $\mathbb{H},$ we obtain that $\eta_0=\varphi(l_0)$ is a hyperbolic geodesic in $U.$ 
		
		We will show that each line $\eta_k,\,k\in\Z,$ is a hyperbolic geodesic in $U.$ First, denote by $\rho_U$ the hyperbolic density in $U.$ By the fact that $U+2\pi i=U,$ we have that
		\begin{equation} \label{hyptr}
			\rho_U(z)=\rho_U(z+2\pi i), \text{ for all } z\in U.
		\end{equation}
		Let $z_1,\,z_2\in U,$ and $\gamma:[0,1]\to U$ be the geodesic arc in $U$ from $z_1$ to $z_2.$ Using \eqref{hyptr}, we can show that $\gamma+2\pi i$ and $\gamma-2\pi i$ are the geodesic arcs in $U$ from $z_1+2\pi i$ to $z_2+2\pi i,$ and from $z_1-2\pi i$ to $z_2-2\pi i,$ respectively. Combining this with the fact that $\eta_0$ is a hyperbolic geodesic in $U,$ we conclude that the lines $\eta_k,\,k\in\Z,$ are hyperbolic geodesics in $U.$
				
		Next, we determine the shape of the hyperbolic geodesics $\varphi^{-1}(\eta_k),\,k\in\Z,$ in $\mathbb{H}.$ The geodesic rays $\eta_k^+,\,k\in\Z,$ belong to the same access to infinity from $U$ as $\eta_0^+,$ whereas the geodesic rays $\eta_k^-,\,k\in\Z,$ belong to distinct accesses to infinity from $U,$ different from the access to which $\eta_0^+$ belongs. Since $\varphi^{-1}(\eta_0^+)=l_{0,\,\alpha}^+,$ \cite[Correspondence Theorem, p. 1838]{BFXK17} implies that $\varphi^{-1}(\eta_k^+)$ lands at $\infty,$ for each $k\in\Z,$ while $\varphi^{-1}(\eta_k^-),\,k\in\Z,$ land at distinct points in $\R.$ It follows that the hyperbolic geodesics $\varphi^{-1}(\eta_k),\,k\in\Z,$ are vertical lines in $\mathbb{H}.$
		
		Now, consider the Riemann map $\Psi: U\to\mathbb{H},$ given by $\Psi(z)=-\overline{\varphi^{-1}(\bar{z})},\,z\in U.$ For all $x\in\R,\,\Psi(x)=-\overline{\varphi^{-1}(\bar{x})}=-\overline{\varphi^{-1}(x)}=\varphi^{-1}(x),$ because $\varphi^{-1}(\eta_0)=l_0.$ By the Identity Theorem, $\varphi^{-1}(z)=\Psi(z)=-\overline{\varphi^{-1}(\bar{z})},$ for all $z\in U,$ which means that, for each $z\in U,$ the points $\varphi^{-1}(z)$ and $\varphi^{-1}(\bar{z})$ are symmetrically located with respect to $l_0$ in $\mathbb{H}.$ Also, it follows from the definition of hyperbolic density that 
		\begin{equation} \label{hypcon}
			\rho_U(z)=\rho_U(\bar{z}), \text{ for all } z\in U.
		\end{equation} 
		
		The last step in the proof of \eqref{varphi} is to show that the function $$z\mapsto\varphi^{-1}(z)-\varphi^{-1}(z+2\pi i)$$ is constant in $U.$ Let $x\in\R,\,\gamma:[0,1]\to U$ be the geodesic arc in $U$ from $x-2\pi i$ to $x,$ and consider the curve $\hat{\gamma}(t)=\overline{\gamma(1-t)},\,t\in[0,1],$ which joins $x$ to $x+2\pi i$ in $U.$ Then \eqref{hypcon} implies that $\hat{\gamma}$ is the geodesic arc in $U$ from $x$ to $x+2\pi i.$ Since $\gamma+2\pi i$ is the geodesic arc in $U$ from $x$ to $x+2\pi i,$ its uniqueness gives
		\begin{equation} \label{symtr}
			\hat{\gamma}([0,1])=\gamma([0,1])+2\pi i.
		\end{equation} 
		It follows that 
		\begin{equation} \label{impart}
			\text{Im}(\varphi^{-1}(x-2\pi i))=\text{Im}(\varphi^{-1}(x))=\text{Im}(\varphi^{-1}(x+2\pi i)).
		\end{equation}
		Indeed, let $\omega$ be the angle between $\eta_{-1}$ and $\gamma$ at the point $x-2\pi i.$ By \eqref{symtr} and the definition of $\hat{\gamma},$ we have that the angle between $\gamma$ and $\eta_0$ at $x,$ the angle between $\eta_0$ and $\hat{\gamma}$ at $x,$ and the one between $\hat{\gamma}$ and $\eta_1$ at $x+2\pi i$ are all equal to $\omega$ (see Figure \ref{angles}). Now, by what we have shown in the two previous paragraphs, the hyperbolic geodesics $\eta_{-1}$ and $\eta_1$ are mapped under $\varphi^{-1}$ onto vertical hyperbolic geodesics in $\mathbb{H},$ which are symmetric with respect to $l_0.$ It is easy to see that $\varphi^{-1}(\gamma)$ and $\varphi^{-1}(\hat{\gamma})$ are circular geodesic arcs in $\mathbb{H},$ symmetric with respect to $l_0$ (see also Figure \ref{angles}). Then, the preservation of angles by $\varphi$ gives \eqref{impart}. Next, we can deduce that $\text{Re}(\varphi^{-1}(2\pi i))<0,$ since $\varphi$ preserves angles and $i\varphi'(i\alpha)>0.$ Set $\kappa=-\text{Re}(\varphi^{-1}(2\pi i))>0.$ Then $\varphi^{-1}(\eta_{-1})$ and $\varphi^{-1}(\eta_1)$ land at $\kappa$ and $-\kappa,$ respectively. This fact, combined with \eqref{impart}, yields that $\varphi^{-1}(x)-\varphi^{-1}(x+2\pi i)=\kappa.$ The latter equality holds for an arbitrarily chosen point $x\in\R,$ so it holds for all $x\in\R.$ By the Identity Theorem, 
		$$\varphi^{-1}(z)-\varphi^{-1}(z+2\pi i)=\kappa,\text{ for all } z\in U.$$
		\begin{figure}[h]
			\centering
			\includegraphics[width=0.95\linewidth]{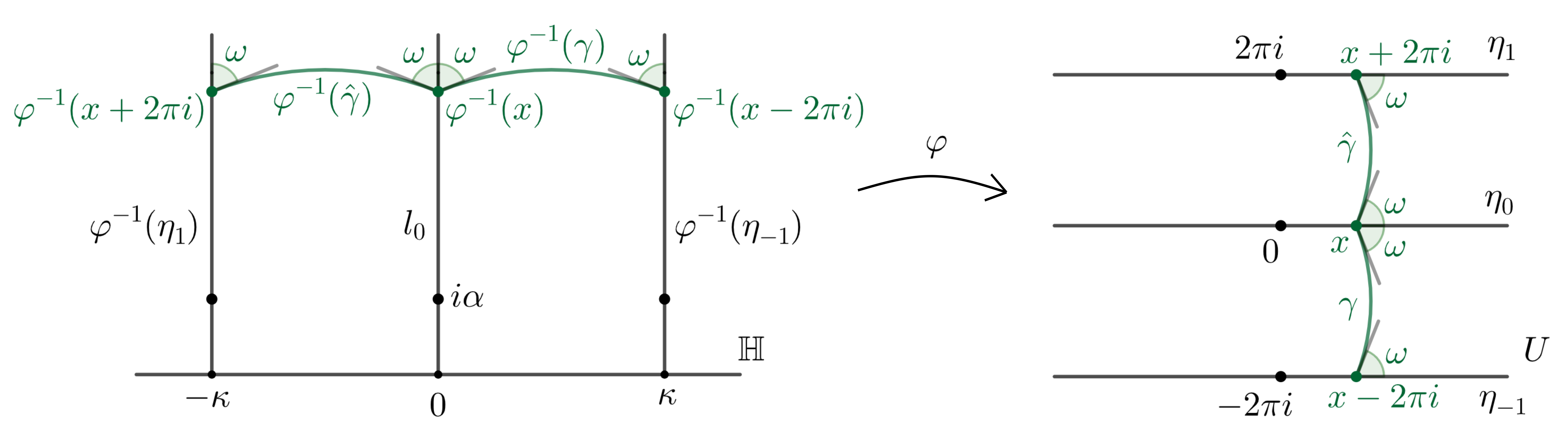}
			\caption{The action of $\varphi$ on $\mathbb{H}$}
			\label{angles}
		\end{figure}
		
		Having established the above results, we can choose $\alpha$ appropriately, so that $\kappa=\pi,$ and hence we obtain \eqref{varphi}. We are now ready to calculate $\sigma$ and $\theta.$ It follows from \eqref{varphi} that $\varphi(l_k)=\eta_{-k},$ with
		\begin{equation} \label{phigen}
			\varphi(k\pi+i\alpha)=-2k\pi i,\,\varphi(l_{k,\,\alpha}^+)=\eta_{-k}^+ \text{ and } \varphi(l_{k,\,\alpha}^-)=\eta_{-k}^-,
		\end{equation}
		for all $k\in\Z$ (see Figure \ref{mhyp}). Also, the Riemann map $\psi:\D\to G,$ from the proof of Theorem \ref{const}, satisfies $\psi(0)=0,$ and we can assume that $\psi'(0)>0,$ so that we have
	    \begin{equation} \label{psi}
	    	\psi((0,\psi^{-1}(1)))=(0,1) \text{ and } \psi((\psi^{-1}(1),1))=(1,+\infty),
	    \end{equation}
	    due to the symmetry of $G$ with respect to the real axis. Let $\beta>0$ be such that $\hat{p}(i\beta)=\psi^{-1}(1),$ where $\hat{p}(z)=e^{2iz},\,z\in\mathbb{H}.$ By the definition of $\hat{p},$
	    \begin{equation} \label{hatp}
	    	\hat{p}(l_{k,\,\beta}^+)=(0,\psi^{-1}(1)) \text{ and } \hat{p}(l_{k,\,\beta}^-)=(\psi^{-1}(1),1),
	    \end{equation}
	    and
	    \begin{equation} \label{hatplim}
	    	\lim_{t\to+\infty}\hat{p}(k\pi+it)=0 \text{ and } \lim_{t\to0^+}\hat{p}(k\pi+it)=1,
	    \end{equation}
	    for all $k\in\Z$ (see Figure \ref{mhyp}). Recall that $M$ is the conformal automorphism of $\mathbb{H}$ such that
	    \begin{equation} \label{defm}
	    	\hat{p}=\psi^{-1}\circ\pi_{-1}\circ\varphi\circ M \text{ in } \mathbb{H}.
	    \end{equation}
	    By \eqref{phigen}-\eqref{defm} and the fact that $M$ maps hyperbolic geodesics in $\mathbb{H}$ onto hyperbolic geodesics in $\mathbb{H},$ we conclude that for each $k\in\Z,$ there exists $j_k\in\Z$ such that $M(l_{j_k,\,\beta}^+)=l_{k,\,\alpha}^+$ and $M(l_{j_k,\,\beta}^-)=l_{k,\,\alpha}^-.$ 
		
		Now, let $\Phi:=\varphi\circ M:\mathbb{H}\to U.$ Then $\hat{g}_f=\Phi^{-1}\circ f\circ\Phi$ in $\mathbb{H}.$ By the above discussion, there exists $j_0\in\Z$ such that
		\begin{equation} \label{j0}
			\Phi(l_{j_0,\,\beta}^+)=(0,+\infty) \text{ and } \Phi(l_{j_0,\,\beta}^-)=(-\infty,0),
		\end{equation} 
		with $\lim_{t\to0^+}\Phi(j_0\pi+it)=-\infty$ and $\lim_{t\to+\infty}\Phi(j_0\pi+it)=+\infty.$ Note that $f(\R)\subset\R,$ with $\lim_{x\to-\infty}f(x)=+\infty,$ hence $$\lim_{t\to0^+}\hat{g}_f(j_0\pi+it)=\lim_{t\to0^+}\Phi^{-1}\circ f\circ\Phi(j_0\pi+it)=\infty.$$ This shows that $j_0\pi$ is a pole of $\hat{g}_f.$ By \eqref{ex1g}, we can deduce that the real poles of $\hat{g}_f$ are the points $z\in\R$ for which $e^{i\theta}-e^{2iz}=0,$ or equivalently, the points $\frac{\theta}{2}+\lambda\pi,\,\lambda\in\Z.$ Thus, $j_0\pi=\frac{\theta}{2}+\lambda\pi,$ for some $\lambda\in\Z,$ hence $\theta=2(j_0-\lambda)\pi,$ and so  $$\hat{g}_f(z)=\sigma+2z+\frac{ic}{2}\cdot\frac{1+e^{2iz}}{1-e^{2iz}}=\sigma+2z-\frac{c}{2}\cdot\cot z,\,z\in\mathbb{H}.$$ Also, $\hat{g}_f(l_{j_0})=\Phi^{-1}\circ f\circ\Phi(l_{j_0})\subset l_{j_0},$ because of \eqref{j0} and the fact that $f(\R)\subset\R,$ and note that $\hat{g}_f(j_0\pi+it)=\sigma+2j_0\pi+2it+\frac{ic}{2}\cdot\frac{1+e^{-2t}}{1-e^{-2t}},$ for all $t\in(0,+\infty).$ Since $\text{Re}\,\hat{g}_f(j_0\pi+it)=j_0\pi,$ for all $t>0,$ we deduce that $\sigma=-j_0\pi.$ By considering the conjugate $g_f:=\hat{M}^{-1}\,\hat{g}_f\,\hat{M}$ to $\hat{g}_f,$ where $\hat{M}(z)=z+j_0\pi,\,z\in\mathbb{H},$ we have $$g_f(z)=2z-\frac{c}{2}\cdot\cot z,$$ for all $z\in\mathbb{H},$ and $g_f$ is a dynamically associated inner function to $f|_U.$ 		
	\end{proof}
	\begin{obs}
		Note that $f'(z)=2-e^{-z},\,z\in\mathbb{C},$ hence $$f'(z)=0 \;\Leftrightarrow\; z=-\log2+2k\pi i\in U,\,k\in\Z,$$ and the critical values of $f$ are $f(-\log2+2k\pi i)=2-2\log2+4k\pi i\in U,\,k\in\Z.$ Thus, $S(f)\cap U$ is not compact, so we cannot use \cite[Theorem 1.2]{ERS20} to find the number of singularities of an inner function dynamically associated to $f|_U.$ Moreover, the technique used in the proof of \cite[Theorem 1.9]{ERS20} cannot be applied, mainly due to the fact that we cannot find the fixed points of $f,$ or the hyperbolic distance between these points.
	\end{obs}
	\begin{rmk}
		Using the dynamically associated inner function to $h|_G,\,g_h,$  which appears in the proof of Theorem \ref{const}, one can show that $\psi'(0)=e^{-c}.$ Then, well-known hyperbolic estimates imply that $c\le0.6.$
	\end{rmk}
	
	The next example is a function that was studied in \cite[Example 3.6]{barg}.
	\begin{eg} \label{ex2}
		Let $f(z)=2z-3+e^z,\,z\in\C,$ and $U$ be the Baker domain of $f.$ Then $g_f(z)=2z+\frac{c}{2}\cdot\tan z,\,z\in\mathbb{H},$ where $c>0,$ is an inner function dynamically associated to $f|_U.$ 
	\end{eg}
	\begin{proof}
		Using \eqref{berl} with $\pi_1(z)=e^z,\,z\in\C,$ we obtain that $f$ is a lift of the function $$h(w)=w^2e^{w-3},\,w\in\C.$$ First we study $h.$ It is easy to see that it has exactly two real fixed points; $0\in F(h)$ is a super-attracting fixed point, and $x_0\in J(h)$ is a repelling fixed point, with $x_0\in(2.2,2.3)$ and $x_0e^{x_0}=e^3.$ The set of singular values of $h$ is $S(h)=\{0,4e^{-5}\},$ where 0 is both a critical value and a finite asymptotic value, and $4e^{-5}$ is a critical value. Since $S(h)$ is finite, $h\in\mathcal{S}\subset\mathcal{B}.$
		
		Note that if $z=x+iy\in\mathbb{C}$ with $x<0.7,$ then $$\text{Re}f(x+iy)=2x-3+e^x\cos y<x.$$ Thus, if $P:=\{z\in\mathbb{C}: \text{Re}z<0.7\},$ then $f(P)\subset P,$ and so $P\subset F(f).$ 
		
		Let $G\subset F(h)$ be the immediate super-attracting basin of 0. It follows from \eqref{fh} that $D(0,e^{0.7})=\pi_1(P)\cup\{0\}\subset G.$ Also, $4e^{-5}\in D(0,e^{0.7}),$ hence \cite[Theorem 1.7]{EFXS19} implies that $G=F(h)$ and $h$ is of disjoint type. One can further observe that $h(\bar{w})=\overline{h(w)},$ for all $w\in\C,$ so $G$ is symmetric with respect to the real axis, and $J(h)$ is a Cantor bouquet, since $h$ is of disjoint type and has order 1 (see \cite[Theorem 1.5]{BJR12}). Finally, by doing some calculations, one can show that $(-\infty,\,x_0)\subset F(h)=G,$ and $[x_0,+\infty)\subset J(h).$
		
		By \eqref{fh}, $U:=\pi_1^{-1}(G)=F(f),$ and working as in Example \ref{ex1}, we show that $U$ is a forward-invariant hyperbolic Baker domain, symmetric with respect to the real axis, $U+2\pi i=U,$ and $f|_U$ has infinite degree. Also, \eqref{fh} implies that $x+(2k+1)\pi i\in F(f),$ for all $x\in\R$ and $k\in\Z,\,x+2k\pi i\in F(f),$ for all $x<\log x_0$ and $k\in\Z,$ and $x+2k\pi i\in J(f),$ for all $x\ge\log x_0$ and $k\in\Z.$ Finally, there exists $\delta\in\big(0,\frac{\pi}{3}\big]$ such that $$J(f)\subset\bigcup_{k\in\Z}\{z\in\C: \text{Re}z\ge0.7,\,(2k-1)\pi+\delta<\text{Im}z<(2k+1)\pi-\delta\},$$ and $J(f)$ consists of disjoint curves tending to $\infty,$ each homeomorphic to $[0,+\infty),$ by \cite[Theorem 1.3]{RRS10}.  
		
		Note that $h|_G$ has infinite degree, so it satisfies the hypotheses of Corollary \ref{infpoly}. Thus, a dynamically associated inner function to $f|_U$ is 
		$$\hat{g}_f(z)=\sigma+2z+\frac{ic}{2}\cdot\frac{e^{i\theta}+e^{2iz}}{e^{i\theta}-e^{2iz}},\,z\in\mathbb{H},$$	
		where $\sigma,\,\theta\in\R$ and $c>0.$ In order to calculate $\sigma$ and $\theta,$ we work as in Example \ref{ex1}, using the same notation. Recall that $\varphi:\mathbb{H}\to U$ is the Riemann map such that $$\varphi^{-1}(z+2\pi i)=\varphi^{-1}(z)-\pi, \text{ for all } z\in U,$$ and $M$ is the conformal automorphism of $\mathbb{H}$ such that $$\hat{p}=\psi^{-1}\circ\pi_1\circ\varphi\circ M \text{ in } \mathbb{H}.$$ The difference in this example is that for each $k\in\Z$ there exists $j_k\in\Z$ such that $$M(l_{j_k,\,\beta}^-)=l_{k,\,\alpha}^+ \text{ and } M(l_{j_k,\,\beta}^+)=l_{k,\,\alpha}^-$$ (for the definition of the $l_{k,\,\xi}^+,\,l_{k,\,\xi}^-,\,l_k$ see \eqref{vertr}, \eqref{vertgeo}). This is due to the fact that $$\pi_1((0,\log x_0))=(1,\,x_0) \text{ and } \pi_1((-\infty,0))=(0,1)$$ (see Figure \ref{pihalf}). 
		\begin{figure}[h]
			\centering
			\includegraphics[width=0.9\linewidth]{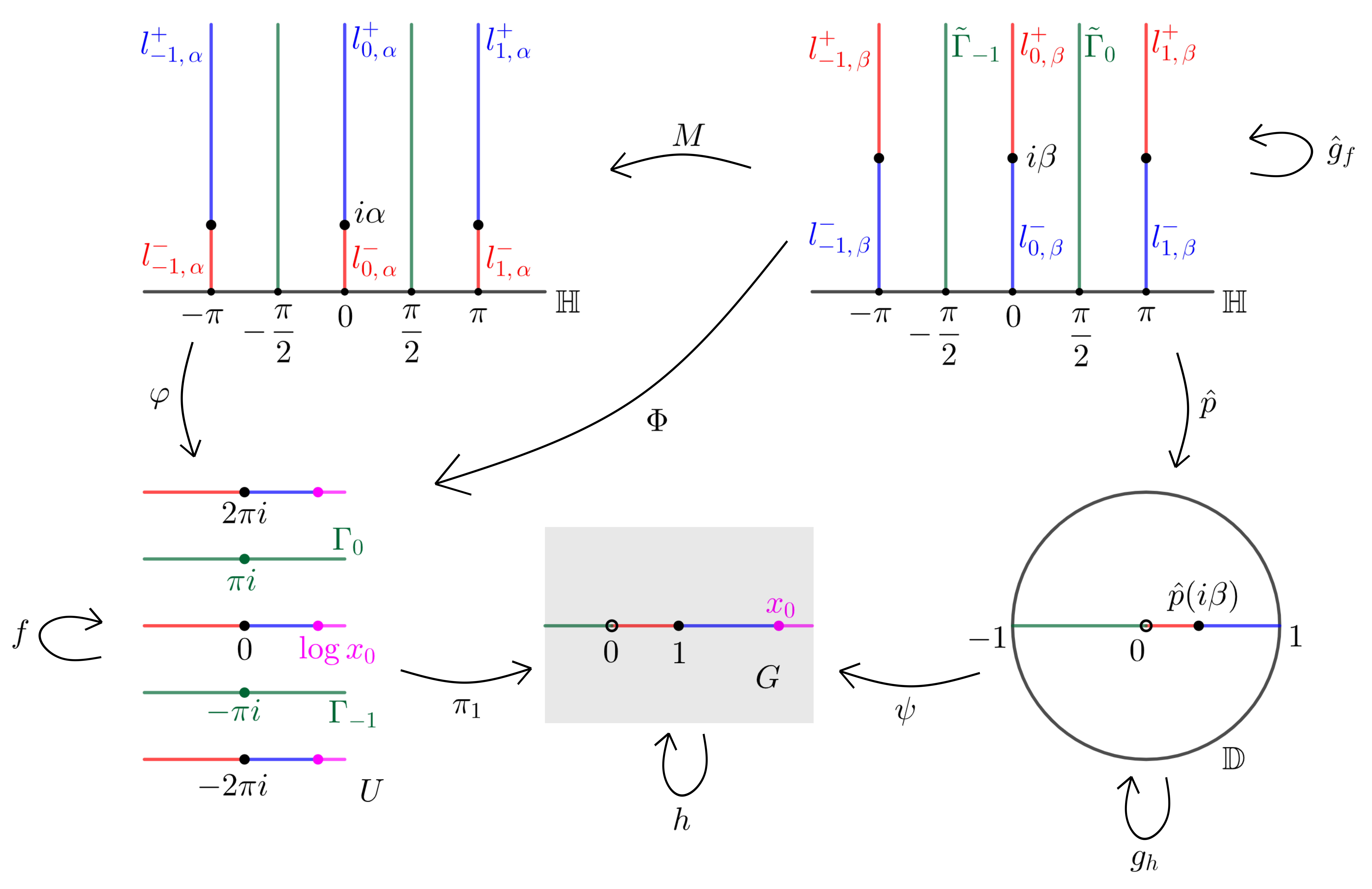}
			\caption{}
			\label{pihalf}
		\end{figure}
		It follows that $$\lim_{t\to0^+}\hat{g}_f(j_0\pi+it)=\lim_{t\to0^+}\Phi^{-1}\circ f\circ\Phi(j_0\pi+it)=j_0\pi,$$ where $j_0\in\Z$ is such that $\Phi(l_{j_0,\,\beta}^+)=(-\infty,0)$ and $\Phi(l_{j_0,\,\beta}^-)=(0,\, \log x_0),$ and $\Phi:\mathbb{H}\to U$ is defined as in Example \ref{ex1}. Thus, we cannot calculate $\theta$ using the above limit, and so we will find the poles of $\hat{g}_f$ by using different hyperbolic geodesics in $\mathbb{H}.$ 
		
		Let $$\tilde{\Gamma}_k:=\Big\{z\in\mathbb{H}: \text{Re}z=(2k+1)\frac{\pi}{2}\Big\} \text{ and } \Gamma_k:=\Big\{z\in\C: \text{Im}z=(2k+1)\pi\Big\},$$ $k\in\Z,$ and note that the hyperbolic geodesic $\tilde{\Gamma}_k$ lies in between the hyperbolic geodesics $l_k$ and $l_{k+1}$ in $\mathbb{H},$ while the horizontal line $\Gamma_k$ lies in between the horizontal lines $y=2k\pi$ and $y=2(k+1)\pi$ in the dynamical plane of $f.$ One can derive from the form of $\hat{p}$ that $\hat{p}(\tilde{\Gamma}_k)=(-1,0),$ with $$\lim_{t\to0^+}\hat{p}\Big((2k+1)\frac{\pi}{2}+it\Big)=-1 \text{ and } \lim_{t\to+\infty}\hat{p}\Big((2k+1)\frac{\pi}{2}+it\Big)=0,$$ for all $k\in\Z.$ Using similar arguments as in Example \ref{ex1}, we deduce that there exists $\nu\in\Z$ such that $\Phi(\tilde{\Gamma}_\nu)=\Gamma_0,$ with $$\lim_{t\to0^+}\Phi\Big((2\nu+1)\frac{\pi}{2}+it\Big)=+\infty+\pi i \text{ and } \lim_{t\to+\infty}\Phi\Big((2\nu+1)\frac{\pi}{2}+it\Big)=-\infty+\pi i$$ (see Figure \ref{pihalf}). Note that $f(x+\pi i)=2x-e^x-3+2\pi i,$ for all $x\in\R,$ hence $$\lim_{t\to0^+}\hat{g}_f\Big((2\nu+1)\frac{\pi}{2}+it\Big)=\lim_{t\to0^+}\Phi^{-1}\circ f\circ\Phi\Big((2\nu+1)\frac{\pi}{2}+it\Big)=\infty.$$ This means that $(2\nu+1)\frac{\pi}{2}$ is a pole of $\hat{g}_f,$ so we can show that $\theta=(2\mu+1)\pi,$ for some $\mu\in\Z,$ working as in Example \ref{ex1}. Thus, $$\hat{g}_f(z)=\sigma+2z+\frac{ic}{2}\cdot\frac{-1+e^{2iz}}{-1-e^{2iz}}=\sigma+2z+\frac{c}{2}\cdot\tan z,\,z\in\mathbb{H}.$$ Also, as in Example \ref{ex1}, $\hat{g}_f(l_{j_0})\subset l_{j_0},$ hence $\sigma=-j_0\pi.$ By considering an appropriate conjugate $g_f$ to $\hat{g}_f,$ we obtain that $$g_f:\mathbb{H}\to\mathbb{H}; \;\;\;g_f(z)=2z+\frac{c}{2}\cdot\tan z,\,z\in\mathbb{H},$$ is a dynamically associated inner function to $f|_U.$
	\end{proof}	 
	\begin{rmk}
		Using the inner function $g_h$ which is dynamically associated to $h|_G$ and appears in the proof of Theorem \ref{const}, one can show that $\psi'(0)=e^{3-c}.$ Then, well-known hyperbolic estimates imply that $0.8<c\le2.3.$
	\end{rmk}
	We finish this section with an example of the finite-degree case, which involves a function discussed in \cite[Example 4]{FH06}. 
	\begin{eg} \label{tex}
		Let $f(z)=z-\frac{1}{2}e^{-2z}+2e^{-z},\,z\in\mathbb{C}.$ Then $f$ has a sequence of forward-invariant Baker domains $\{U_k\}_{k\in\mathbb{Z}},$ and $z\mapsto\frac{2z^3+1}{2+z^3},\,z\in\mathbb{D},$ is an inner function dynamically associated to $f|_{U_k},$ for each $k\in\Z.$ 
	\end{eg}
	\begin{proof}
		Using \eqref{berl} with $\pi_{-1}(z)=e^{-z},\,z\in\C,$ we derive that $f$ is a lift of the function $$h(w)=we^{\frac{1}{2}w^2-2w},\,w\in\C.$$ The function $h$ has two real fixed points; 0 is a parabolic fixed point, with one attracting petal attached to it (for more details on the dynamics near a parabolic fixed point see, for example, \cite[§10]{milnor}), and 4 is a repelling fixed point. Let $G\subset F(h)$ be the immediate parabolic basin of 0 (so $0\in\partial G$ and $G$ is forward-invariant). An elementary study of $h$ shows that $S(h)=\{0, e^{-3/2}\},$ so $h\in\mathcal{S}\subset\mathcal{B},$ and $(0,4)\subset G,$ while $(-\infty,0]\cup[4,+\infty)\subset J(h).$
		
		Regarding the dynamical plane of $f,$ the horizontal lines $l_k=\{z\in\C: \text{Im}z=(2k+1)\pi\},$ $k\in\Z,$ are in $J(f),$ because of \eqref{fh} and the fact that $\pi_{-1}(\l_k)=(-\infty,0)\subset J(h),$ for all $k\in\Z.$ Next, let $\{U_k\}_{k\in\Z}$ be the connected components of $\pi_{-1}^{-1}(G)$ (see Corollary \ref{liftgl}(ii)). Then $U_{k+1}=U_k-2\pi i,$ for all $k\in\Z,$ so we can assume, without loss of generality, that $U_0\subset\{z\in\C: -\pi<\text{Im}z<\pi\}.$ Since $0,\,f(0)=\frac{3}{2}\in U_0$ (due to \eqref{fh}), we derive that $U_0$ is forward-invariant and hence, each $U_k$ is forward-invariant by Remark \ref{wdpli}(1)(b). Also, \eqref{berl} implies that each $U_k$ is a Baker domain of $f.$ 
		
		It is clear that $h$ satisfies the assumptions of Corollary \ref{finpoly}, hence, if $g_h$ is a dynamically associated inner function to $h|_G,$ then $g_h$ is a finite Blaschke product of degree 3, and it is dynamically associated to $f|_{U_k},$ for all $k\in\Z.$ In order to calculate $g_h$ explicitly, we use the fact that it is dynamically associated to $U_0$ (we chose $U_0$ because it is symmetric with respect to the real axis). Note that $g_h$ is doubly-parabolic (for the definition of a doubly-parabolic self-map of $\D$ see, for example, \cite[Theorem 2.1]{Jov24}), as it is dynamically associated to a parabolic Fatou component (see \cite[Theorem 10.9, Corollary 10.13]{milnor}). Working as in the first part of the proof of \cite[Theorem 5.2]{BD99}, and taking into account that 0 is the only critical point of $f|_{U_0}$ and $f$ has local degree 3 at 0, we can show that, for an appropriate choice of the Riemann map, $g_h(z)=\frac{z^3+\alpha}{1+\alpha z^3},$ where $\alpha=g_h(0)\in(0,1).$ In order to calculate $\alpha,$ we use the fact that 1 is the Denjoy-Wolff point of $g_h$ and $g_h$ is doubly-parabolic, which imply that $g_h'(1)=1.$ Thus, $\alpha=\frac{1}{2}.$ 
	\end{proof}
	\printbibliography

@article {EFXS19,
    AUTHOR = {Evdoridou, V. and Fagella, N. and Jarque, X.
              and Sixsmith, D. J.},
     TITLE = {Singularities of inner functions associated with hyperbolic
              maps},
   JOURNAL = {J. Math. Anal. Appl.},
  FJOURNAL = {Journal of Mathematical Analysis and Applications},
    VOLUME = {477},
      YEAR = {2019},
    NUMBER = {1},
     PAGES = {536--550},
      ISSN = {0022-247X,1096-0813},
   MRCLASS = {30J05 (30D05 30D15 37F10)},
  MRNUMBER = {3950051},
MRREVIEWER = {Crist\'obal\ Gonz\'alez},
       DOI = {10.1016/j.jmaa.2019.04.045},
       URL = {https://doi.org/10.1016/j.jmaa.2019.04.045},
}

@article {BerFR15,
    AUTHOR = {Bergweiler, W. and Fagella, N. and Rempe-Gillen,
              L.},
     TITLE = {Hyperbolic entire functions with bounded {F}atou components},
   JOURNAL = {Comment. Math. Helv.},
  FJOURNAL = {Commentarii Mathematici Helvetici. A Journal of the Swiss
              Mathematical Society},
    VOLUME = {90},
      YEAR = {2015},
    NUMBER = {4},
     PAGES = {799--829},
      ISSN = {0010-2571,1420-8946},
   MRCLASS = {37F10 (30D05 37F15)},
  MRNUMBER = {3433280},
MRREVIEWER = {David\ Jonathon\ Sixsmith},
       DOI = {10.4171/CMH/371},
       URL = {https://doi.org/10.4171/CMH/371},
}

@article {ERS20,
    AUTHOR = {Evdoridou, V. and Rempe, L. and Sixsmith, D. J.},
     TITLE = {Fatou's associates},
   JOURNAL = {Arnold Math. J.},
  FJOURNAL = {Arnold Mathematical Journal},
    VOLUME = {6},
      YEAR = {2020},
    NUMBER = {3-4},
     PAGES = {459--493},
      ISSN = {2199-6792,2199-6806},
   MRCLASS = {37F10 (30D05 30J05 30J10)},
  MRNUMBER = {4181721},
MRREVIEWER = {Kirill\ Lazebnik},
       DOI = {10.1007/s40598-020-00148-6},
       URL = {https://doi.org/10.1007/s40598-020-00148-6},
}

@article {B84,
    AUTHOR = {Baker, I. N.},
     TITLE = {Wandering domains in the iteration of entire functions},
   JOURNAL = {Proc. London Math. Soc. (3)},
  FJOURNAL = {Proceedings of the London Mathematical Society. Third Series},
    VOLUME = {49},
      YEAR = {1984},
    NUMBER = {3},
     PAGES = {563--576},
      ISSN = {0024-6115,1460-244X},
   MRCLASS = {58F11 (30D05)},
  MRNUMBER = {759304},
MRREVIEWER = {Robert\ L.\ Devaney},
       DOI = {10.1112/plms/s3-49.3.563},
       URL = {https://doi.org/10.1112/plms/s3-49.3.563},
}

@article {BFXK17,
    AUTHOR = {Bara\'nski, K. and Fagella, N. and Jarque, X.
              and Karpi\'nska, B.},
     TITLE = {Accesses to infinity from {F}atou components},
   JOURNAL = {Trans. Amer. Math. Soc.},
  FJOURNAL = {Transactions of the American Mathematical Society},
    VOLUME = {369},
      YEAR = {2017},
    NUMBER = {3},
     PAGES = {1835--1867},
      ISSN = {0002-9947,1088-6850},
   MRCLASS = {37F10 (30D05 30D30)},
  MRNUMBER = {3581221},
MRREVIEWER = {Tarakanta\ Nayak},
       DOI = {10.1090/tran/6739},
       URL = {https://doi.org/10.1090/tran/6739},
}

@article {BJR12,
    AUTHOR = {Bara\'nski, K. and Jarque, X. and Rempe, L.},
     TITLE = {Brushing the hairs of transcendental entire functions},
   JOURNAL = {Topology Appl.},
  FJOURNAL = {Topology and its Applications},
    VOLUME = {159},
      YEAR = {2012},
    NUMBER = {8},
     PAGES = {2102--2114},
      ISSN = {0166-8641,1879-3207},
   MRCLASS = {37F10 (28A80 30D05 37F50 54F15 54H20)},
  MRNUMBER = {2902745},
MRREVIEWER = {Anand\ Prakash\ Singh},
       DOI = {10.1016/j.topol.2012.02.004},
       URL = {https://doi.org/10.1016/j.topol.2012.02.004},
}

@article {Ber95,
    AUTHOR = {Bergweiler, W.},
     TITLE = {On the {J}ulia set of analytic self-maps of the punctured
              plane},
   JOURNAL = {Analysis},
  FJOURNAL = {Analysis. International Mathematical Journal of Analysis and
              its Applications},
    VOLUME = {15},
      YEAR = {1995},
    NUMBER = {3},
     PAGES = {251--256},
      ISSN = {0174-4747},
   MRCLASS = {30D05 (30D15 30D45)},
  MRNUMBER = {1357963},
MRREVIEWER = {Eugen\ Mihailescu},
       DOI = {10.1524/anly.1995.15.3.251},
       URL = {https://doi.org/10.1524/anly.1995.15.3.251},
}

@article {EL92,
    AUTHOR = {Er\"emenko, A. \`E. and Lyubich, M. Yu.},
     TITLE = {Dynamical properties of some classes of entire functions},
   JOURNAL = {Ann. Inst. Fourier (Grenoble)},
  FJOURNAL = {Universit\'e{} de Grenoble. Annales de l'Institut Fourier},
    VOLUME = {42},
      YEAR = {1992},
    NUMBER = {4},
     PAGES = {989--1020},
      ISSN = {0373-0956,1777-5310},
   MRCLASS = {30D05 (58F23)},
  MRNUMBER = {1196102},
MRREVIEWER = {Ben\ Bielefeld},
       DOI = {10.5802/aif.1318},
       URL = {https://doi.org/10.5802/aif.1318},
}

@incollection {barg,
    AUTHOR = {Bargmann, D.},
     TITLE = {Iteration of inner functions and boundaries of components of
              the {F}atou set},
 BOOKTITLE = {Transcendental dynamics and complex analysis},
    SERIES = {London Math. Soc. Lecture Note Ser.},
    VOLUME = {348},
     PAGES = {1--36},
 PUBLISHER = {Cambridge Univ. Press, Cambridge},
      YEAR = {2008},
      ISBN = {978-0-521-68372-2},
   MRCLASS = {30D05 (37F10 37F50)},
  MRNUMBER = {2458797},
MRREVIEWER = {Peter\ Ha\"issinsky},
       DOI = {10.1017/CBO9780511735233.003},
       URL = {https://doi.org/10.1017/CBO9780511735233.003},
}

@article {KU,
    AUTHOR = {Kotus, J. and Urba\'nski, M.},
     TITLE = {The dynamics and geometry of the {F}atou functions},
   JOURNAL = {Discrete Contin. Dyn. Syst.},
  FJOURNAL = {Discrete and Continuous Dynamical Systems. Series A},
    VOLUME = {13},
      YEAR = {2005},
    NUMBER = {2},
     PAGES = {291--338},
      ISSN = {1078-0947,1553-5231},
   MRCLASS = {37F35 (37D35 37F10 37F15)},
  MRNUMBER = {2152392},
MRREVIEWER = {Peter\ Ha\"issinsky},
       DOI = {10.3934/dcds.2005.13.291},
       URL = {https://doi.org/10.3934/dcds.2005.13.291},
}

@article {BD99,
    AUTHOR = {Baker, I. N. and Dom\'inguez, P.},
     TITLE = {Boundaries of unbounded {F}atou components of entire
              functions},
   JOURNAL = {Ann. Acad. Sci. Fenn. Math.},
  FJOURNAL = {Annales Academi\ae\ Scientiarum Fennic\ae. Mathematica},
    VOLUME = {24},
      YEAR = {1999},
    NUMBER = {2},
     PAGES = {437--464},
      ISSN = {1239-629X,1798-2383},
   MRCLASS = {37F10 (30D05 30D40 37F50)},
  MRNUMBER = {1724391},
MRREVIEWER = {Walter\ Bergweiler},
}

@article {Fatou,
    AUTHOR = {Fatou, P.},
     TITLE = {Sur l'it\'eration des fonctions transcendantes enti\`eres},
   JOURNAL = {Acta Math.},
  FJOURNAL = {Acta Mathematica},
    VOLUME = {47},
      YEAR = {1926},
    NUMBER = {4},
     PAGES = {337--370},
      ISSN = {0001-5962,1871-2509},
   MRCLASS = {99-04},
  MRNUMBER = {1555220},
       DOI = {10.1007/BF02559517},
       URL = {https://doi.org/10.1007/BF02559517},
}

@article {Ber93,
    AUTHOR = {Bergweiler, W.},
     TITLE = {Iteration of meromorphic functions},
   JOURNAL = {Bull. Amer. Math. Soc. (N.S.)},
  FJOURNAL = {American Mathematical Society. Bulletin. New Series},
    VOLUME = {29},
      YEAR = {1993},
    NUMBER = {2},
     PAGES = {151--188},
      ISSN = {0273-0979,1088-9485},
   MRCLASS = {30D05 (58F23)},
  MRNUMBER = {1216719},
MRREVIEWER = {I.\ N.\ Baker},
       DOI = {10.1090/S0273-0979-1993-00432-4},
       URL = {https://doi.org/10.1090/S0273-0979-1993-00432-4},
}

@article {BeniEFRS,
    AUTHOR = {Benini, A. M. and Evdoridou, V. and Fagella,
              N. and Rippon, P. J. and Stallard, G. M.},
     TITLE = {Classifying simply connected wandering domains},
   JOURNAL = {Math. Ann.},
  FJOURNAL = {Mathematische Annalen},
    VOLUME = {383},
      YEAR = {2022},
    NUMBER = {3-4},
     PAGES = {1127--1178},
      ISSN = {0025-5831,1432-1807},
   MRCLASS = {37F10},
  MRNUMBER = {4458398},
       DOI = {10.1007/s00208-021-02252-0},
       URL = {https://doi.org/10.1007/s00208-021-02252-0},
}

@article {BarK07,
    AUTHOR = {Bara\'nski, K. and Karpi\'nska, B.},
     TITLE = {Coding trees and boundaries of attracting basins for some
              entire maps},
   JOURNAL = {Nonlinearity},
  FJOURNAL = {Nonlinearity},
    VOLUME = {20},
      YEAR = {2007},
    NUMBER = {2},
     PAGES = {391--415},
      ISSN = {0951-7715,1361-6544},
   MRCLASS = {37F10 (30D05 30D40 37F50)},
  MRNUMBER = {2290468},
MRREVIEWER = {Rich\ L.\ Stankewitz},
       DOI = {10.1088/0951-7715/20/2/008},
       URL = {https://doi.org/10.1088/0951-7715/20/2/008},
}

@article {JoF25,
    AUTHOR = {Jov\'e, A. and Fagella, N.},
     TITLE = {Boundary dynamics in unbounded {F}atou components},
   JOURNAL = {Trans. Amer. Math. Soc.},
  FJOURNAL = {Transactions of the American Mathematical Society},
    VOLUME = {378},
      YEAR = {2025},
    NUMBER = {4},
     PAGES = {2321--2362},
      ISSN = {0002-9947,1088-6850},
   MRCLASS = {37F10 (30D05 37F12)},
  MRNUMBER = {4880450},
       DOI = {10.1090/tran/9287},
       URL = {https://doi.org/10.1090/tran/9287},
}

@article {herman,
    AUTHOR = {Herman, M. R.},
     TITLE = {Exemples de fractions rationnelles ayant une orbite dense sur
              la sph\`ere de {R}iemann},
   JOURNAL = {Bull. Soc. Math. France},
  FJOURNAL = {Bulletin de la Soci\'et\'e{} Math\'ematique de France},
    VOLUME = {112},
      YEAR = {1984},
    NUMBER = {1},
     PAGES = {93--142},
      ISSN = {0037-9484},
   MRCLASS = {58F08 (30D05 58F11)},
  MRNUMBER = {771920},
MRREVIEWER = {I.\ N.\ Baker},
       URL = {http://www.numdam.org/item?id=BSMF_1984__112__93_0},
}

@article {RRS10,
    AUTHOR = {Rempe, L. and Rippon, P. J. and Stallard, G. M.},
     TITLE = {Are {D}evaney hairs fast escaping?},
   JOURNAL = {J. Difference Equ. Appl.},
  FJOURNAL = {Journal of Difference Equations and Applications},
    VOLUME = {16},
      YEAR = {2010},
    NUMBER = {5-6},
     PAGES = {739--762},
      ISSN = {1023-6198,1563-5120},
   MRCLASS = {37F10 (30D05)},
  MRNUMBER = {2675603},
MRREVIEWER = {Anand\ Prakash\ Singh},
       DOI = {10.1080/10236190903282824},
       URL = {https://doi.org/10.1080/10236190903282824},
}

@article {RS99,
    AUTHOR = {Rippon, P. J. and Stallard, G. M.},
     TITLE = {Families of {B}aker domains. {II}},
   JOURNAL = {Conform. Geom. Dyn.},
  FJOURNAL = {Conformal Geometry and Dynamics. An Electronic Journal of the
              American Mathematical Society},
    VOLUME = {3},
      YEAR = {1999},
     PAGES = {67--78},
      ISSN = {1088-4173},
   MRCLASS = {37F10 (30D05 37F50)},
  MRNUMBER = {1689255},
MRREVIEWER = {Walter\ Bergweiler},
       DOI = {10.1090/S1088-4173-99-00045-4},
       URL = {https://doi.org/10.1090/S1088-4173-99-00045-4},
}

@article {K89,
    AUTHOR = {Keen, L.},
     TITLE = {Topology and growth of a special class of holomorphic
              self-maps of {${\bf C}^*$}},
   JOURNAL = {Ergodic Theory Dynam. Systems},
  FJOURNAL = {Ergodic Theory and Dynamical Systems},
    VOLUME = {9},
      YEAR = {1989},
    NUMBER = {2},
     PAGES = {321--328},
      ISSN = {0143-3857,1469-4417},
   MRCLASS = {30D05 (32G15 58F08)},
  MRNUMBER = {1007413},
MRREVIEWER = {Feliks\ Przytycki},
       DOI = {10.1017/S0143385700004995},
       URL = {https://doi.org/10.1017/S0143385700004995},
}

@book{zakeri,
  title={A course in complex analysis},
  author={Zakeri, S.},
  year={2021},
  publisher={Princeton University Press}
}

@article {topf,
    AUTHOR = {T\"opfer, H.},
     TITLE = {\"Uber die {I}teration der ganzen transzendenten {F}unktionen,
              insbesondere von {$\sin z$} und {$\cos z$}},
   JOURNAL = {Math. Ann.},
  FJOURNAL = {Mathematische Annalen},
    VOLUME = {117},
      YEAR = {1939},
     PAGES = {65--84},
      ISSN = {0025-5831,1432-1807},
   MRCLASS = {30.0X},
  MRNUMBER = {1293},
MRREVIEWER = {G.\ Valiron},
       DOI = {10.1007/BF01450008},
       URL = {https://doi.org/10.1007/BF01450008},
}

@article {dg87,
    AUTHOR = {Devaney, R. L. and Goldberg, L. R.},
     TITLE = {Uniformization of attracting basins for exponential maps},
   JOURNAL = {Duke Math. J.},
  FJOURNAL = {Duke Mathematical Journal},
    VOLUME = {55},
      YEAR = {1987},
    NUMBER = {2},
     PAGES = {253--266},
      ISSN = {0012-7094,1547-7398},
   MRCLASS = {30D05 (58F08)},
  MRNUMBER = {894579},
MRREVIEWER = {I.\ N.\ Baker},
       DOI = {10.1215/S0012-7094-87-05513-X},
       URL = {https://doi.org/10.1215/S0012-7094-87-05513-X},
}

@article {BE95,
    AUTHOR = {Bergweiler, W. and Eremenko, A.},
     TITLE = {On the singularities of the inverse to a meromorphic function
              of finite order},
   JOURNAL = {Rev. Mat. Iberoamericana},
  FJOURNAL = {Revista Matem\'atica Iberoamericana},
    VOLUME = {11},
      YEAR = {1995},
    NUMBER = {2},
     PAGES = {355--373},
      ISSN = {0213-2230},
   MRCLASS = {30D30},
  MRNUMBER = {1344897},
MRREVIEWER = {J.\ Clunie},
       DOI = {10.4171/RMI/176},
       URL = {https://doi.org/10.4171/RMI/176},
}

@book{ivers,
  title={Recherches sur les fonctions inverses des fonctions m{\'e}romorphes},
  author={Iversen, F.},
  url={https://books.google.co.uk/books?id=C2fQAAAAMAAJ},
  year={1914},
  publisher={Imprimerie de la Soci{\'e}t{\'e} de litt{\'e}rature finnoise}
}

@misc{Jov24,
      title={Boundaries of hyperbolic and simply parabolic Baker domains}, 
      author={A. Jové},
      year={2024},
      eprint={2410.19726},
      archivePrefix={arXiv},
      primaryClass={math.DS},
      note={Preprint, arXiv:2410.19726v1},
}

@misc{E21,
      title={Singularities of inverse functions}, 
      author={A. Eremenko},
      year={2021},
      eprint={2110.06134},
      archivePrefix={arXiv},
      primaryClass={math.CV},
      note={arXiv:2110.06134v1},
}

@article {FH06,
    AUTHOR = {Fagella, N. and Henriksen, Ch.},
     TITLE = {Deformation of entire functions with {B}aker domains},
   JOURNAL = {Discrete Contin. Dyn. Syst.},
  FJOURNAL = {Discrete and Continuous Dynamical Systems. Series A},
    VOLUME = {15},
      YEAR = {2006},
    NUMBER = {2},
     PAGES = {379--394},
      ISSN = {1078-0947,1553-5231},
   MRCLASS = {37F10 (30D20 37F50)},
  MRNUMBER = {2199435},
MRREVIEWER = {Walter\ Bergweiler},
       DOI = {10.3934/dcds.2006.15.379},
       URL = {https://doi.org/10.3934/dcds.2006.15.379},
}

@book {milnor,
    AUTHOR = {Milnor, J.},
     TITLE = {Dynamics in one complex variable},
    SERIES = {Annals of Mathematics Studies},
    VOLUME = {160},
   EDITION = {Third},
 PUBLISHER = {Princeton University Press, Princeton, NJ},
      YEAR = {2006},
     PAGES = {viii+304},
      ISBN = {978-0-691-12488-9; 0-691-12488-4},
   MRCLASS = {37Fxx (30-01 30D05 37-01)},
  MRNUMBER = {2193309},
}

@unknown{FJ25,
author = {Rodrigues Ferreira, G. and Jové, A.},
year = {2025},
month = {10},
pages = {},
title = {Boundaries of multiply connected Fatou components. A unified approach},
doi = {10.48550/arXiv.2510.09241},
note={Preprint, arXiv:2510.09241v1},
}
\end{document}